\documentclass{amsart}
\usepackage{mdframed}
\usepackage{dino}
\usepackage{enumitem}
\newmdtheoremenv[backgroundcolor=cyan]{theorem-prove}{Theorem}[theorem]
\newmdtheoremenv[backgroundcolor=cyan]{lemma-prove}{Lemma}[theorem]
\newmdtheoremenv[backgroundcolor=cyan]{proposition-prove}{Proposition}[theorem]

\newmdtheoremenv[backgroundcolor=yellow!40]{theorem-check}{Theorem}[theorem]
\newmdtheoremenv[backgroundcolor=yellow!40]{lemma-check}{Lemma}[theorem]
\newmdtheoremenv[backgroundcolor=yellow!40]{proposition-check}{Proposition}[theorem]

\newcommand{\optp}{\operatorname{tp}}
\newcommand{\norm}[1]{\left\lVert#1\right\rVert}

\newcommand{\tA}{{\bar {\mathcal{A}}}}
\newcommand{\tB}{{\bar{\mathcal{B}}}}

\newcommand{\bbQ}{\mathbb{Q}}

\newcommand*\interior[1]{#1^{\mathsf{o}}}
\newcommand{\bone}{\mathbf{1}}
\newcommand{\bzero}{\mathbf{0}}

\title{Definability and Scott rank in Separable Metric Structures}
\author{Diego Bejarano}
\date{}

\begin{document}

\begin{abstract}
We give a notion of Scott rank for separable metric structures based on the definability of the (metric closures of) automorphism orbits in continuous infinitary logic. This is a continuous analogue of work of Montalb\'an for countable structures. In the process, we prove some results concerning definability, type omitting, and back-and-forth for metric structures.
\end{abstract}

\maketitle

\tableofcontents

\section{Introduction}
In \cite{Scott}, Scott shows that the infinitary logic $\mathcal{L}_{\omega_1,\omega}$ is strong enough to classify countable structures in a particular language up to isomorphism. In particular, for every countable structure $\mathcal{A}$, there is an $\mathcal{L}_{\omega_1,\omega}$ sentence $\phi$ such that for any other countable structure $\mathcal{B}$ in the same language, $\mathcal{A}\cong\mathcal{B}$ if, and only if, $\mathcal{B}\models\phi$. In relation to this result, many notions of Scott rank have been proposed, which aim to give a notion of complexity to countable structures. Of particular note are the notions introduced by Montalb\'an in  \cite{montalban_robuster_2015}, but best reference in Montalb\'an's book \cite{montalban2021}: the Scott rank of a countable structure $\mathcal{A}$ is the least ordinal $\alpha$ such that the automorphism orbit of every finite tuple of elements of $\mathcal{A}$ is $\Sigma_\alpha$ in $\mathcal{L}_{\omega_1,\omega}$ without parameters. Equivalently, it is the least ordinal $\alpha$ such that $\mathcal{A}$ has a $\Pi_{\alpha+1}$ Scott sentence. We refer the reader to \cite[Section 3.1]{montalban_robuster_2015} for a brief summary of the history of Scott rank.\par

Scott's original result was generalized to separable metric structures in \cite{benyaacov2017} using a continuous analogue of Scott's back-and-forth relations, which focuses on how complex are the formulas that define automorphism orbits inside a countable structure. The purpose of this work is to generalize Montalb\'an's notion of Scott rank to separable metric structures. To do this, we investigate the definable sets (without parameters) in separable metric structures to obtain the following results:

\newtheorem*{thmx}{Theorem}

\begin{thmx}
Let $\tA$ be a separable metric structure, $n$ a positive integer, and $\bar a\in\tA^{n}$:
\begin{enumerate}
\item The (metric closure of) the automorphism orbit of $\bar a$ in $\tA$ is definable (without parameters) by an infinitary continuous formula (\Cref{def-orbits}).
\item Therefore, a nonempty subset of $A^n$ is definable (without parameters)  by an infinitary continuous formula if, and only if, it is metric-closed and automorphism-invariant (\Cref{def-sets}).
\item Moreover, under some mild assumptions, the Scott sentence of $\tA$ can be constructed in a uniform way from these definitions (\Cref{SP-exists}). 
\end{enumerate}
\end{thmx}

The main tool used in these results is a metric version of the back-and-forth process, which we develop in Section 4 to characterize isomorphisms of separable metric structures. Using a notion of quantifier complexity defined in Section 2, we define a notion of Scott rank in Section 6 and show that this version of Scott rank satisfies a notion of Robustness similar to the discrete case. However, due to some subtleties with our type-omitting theorem in Section 3, we can only show that the existence of a Scott sentence bounds the Scott rank at limit ordinals (the details are found in Section 6).\par

This paper is organizes as follows: Section 2 introduces the main definition of continuous infinitary logic, and we show basic results about definability and types. Section 3 proves uses a continuous analogue of Makkai's consistency properties to show a Model Existence Theory and a weaker analogue of Montalb\'an's type omitting theorem. Section 4 defines back-and-forth sets and shows that they corresponds to isomorphisms. Section 5 combines the work of the previous sections to show (the metric closures of) automorphism orbits are definable and builds a Scott predicate using these definitions. In Section 6 we put everything together to show some properties of the Scott rank.\par

\textbf{Acknowledgements:} This paper is part of the author's Ph.D. thesis under Tom Scanlon. The author would like to thank Tom Scanlon, Antonio Montalb\'an, Dino Rossegger, Aaron Anderson, Ita{\"i} Ben Yaacov, and C. Ward Henson for many insightful conversations and comments in the making of this paper.

\section{Continuous Infinitary Logic}

In this section we introduce the basic notions of continuous infinitary logic, first defined by Ben Yaacov and Iovino in \cite{BenIovino}. Further work on the topic can be found in \cite{eagle}, \cite{benyaacov2017} and \cite{hallback}. This logic is `continuous'  in the sense that our formulas will be interpreted as uniformly continuous functions into the real numbers, and `infinitary' in that we will add the connectives $\sup_n$ and $\inf_n$, which are interpreted as taking the supremum (resp. infimum)  of a countable set of formulas when certain conditions are met. Our definitions of the logic are essentially those of \cite{hallback}.\par

Throughout most of this paper, $\sup_x$ acts as the continuous analogue of the universal quantifier, and $\sup_n$ acts as the analogue of conjunction. Similarly, $\inf_x$ plays the role of the existential quantifier, and $\inf_n$ the role of countable disjunction. This is because we will be interested primarily in conditions of the form $``\phi<r,"$ where $\sup$ and $\inf$ do take these prescribed roles. However, this convention no longer holds if we look at expressions of the form $``\phi>r"$. Therefore, to avoid confusion, we will not use $\land$, $\lor$, $\bigwedge$ or $\bigvee$ going forward. We begin with the notion of a vocabulary (sometimes also called a signature):

\begin{definition}
By a continuous vocabulary $\tau$, we mean a family of formal symbols. For each symbol $S\in\tau$ we specify:
\begin{enumerate}
\item its kind: either predicate or function. This determines what kind of object its interpretation should be;
\item its arity: a number $n_S$, which determines how many arguments $S$ takes;
\item a function $\Delta_S$, which is a modulus of continuity of arity $n_S$, and we require that any interpretation of $S$ respects this modulus. We will elaborate on this point below;
\item If $S$ is a predicate symbol, we also specify a compact subset of $I_S$ of $\mathbb{R}$. We require that that any interpretation of $S$ has its image contained in this subset.
\end{enumerate}
In this setup, a  constant will be a zero-ary function. We assume that $\tau$ always includes a binary predicate symbol $d$, with $\Delta_d(r_1,r_2)=r_1+r_2$.
\end{definition}

The modulus of continuity is meant as a guarantee from the language that the interpretation of the symbol $S$ will be continuous (and how continuous). For predicates, $\Delta_S$ and $I_S$ together guarantee that the interpretation of $S$ will be a uniformly continuous function. We now give the formal definition:

\begin{definition}
\begin{enumerate}
\item A modulus of continuity of arity $n$ is a continuous function $\Delta:[0,\infty)^n\to[0,\infty)$ such that $\Delta(0)=0$ and
$$\Delta(\overline r)\leq \Delta(\overline r+\overline s)\leq \Delta(\overline r)+\Delta(\overline s),$$
for all $\overline r, \overline s\in[0,\infty)^n$. We say that $\Delta$ is faithful if $\Delta(\overline r)=0$ implies $\overline r = \overline 0$. 
\item Given a metric space $(X,d_X)$, and $\Delta$ a modulus of continuity of arity $n$. We define a pseudo-distance $d_\Delta$ on $X^n$ by letting $d_\Delta(\overline x,\overline y) = \Delta(d_X(x_i,y_i)\mid i<n)$. If $\Delta$ is faithful, $d_\Delta$ is a distance function on $X^n$.
\item We say that a function $S:(X,d_X)^n \to (Z,d_Z)$ between metric spaces respects $\Delta$ if for every $\overline x,\overline y\in X^n$ we have that  $d_Z(S(\overline x),S(\overline y))\leq d_\Delta(\overline x,\overline y)$.
\end{enumerate}
\end{definition}

We can now define the notion of a metric $\tau$-structure:

\begin{definition}
A metric $\tau$-structure $\tA$ consists of a set $A$, known as the domain of the structure, alongside interpretations of all the symbols in $\tau$ such that:
\begin{enumerate}
\item each predicate symbol $P\in \tau$ interprets a bounded, uniformly continuous function $P:A^{n_P}\to \mathbb{R}$ that respects $\Delta_P$ and whose image is contained in $I_P$;
\item $d$ interprets a distance on $A$ such that $(A,d)$ is a complete, bounded metric space;
\item each function symbol $f\in \tau$ interprets a continuous function $f:A^{n_f}\to A$ that respects the modulus $\Delta_F$.
\end{enumerate}
\end{definition}

\begin{example}
\begin{enumerate}[label=\emph{\alph*})]
\item If $B$ is the unit ball of a Banach space $X$ over $\mathbb{R}$ or $\mathbb{C}$. We can name the point $0_X\in B$, and the norm $||\cdot||:B\to [0,1]$ as a predicate. We can also add functions $f_{\alpha,\beta}(x,y) = \alpha x+\beta y$ for each pair of scalars with $|\alpha|+|\beta|\leq 1$. Alternatively, there is way of treating all of $X$ as a metric structure by using countably many sorts to name the ball of radius $n$ centered at  $0_X\in B$ for every $n\in\mathbb{N}$. To capture the structure of a $C^*$ algebra, we can include multiplication and the $^*$-map as functions. 
\item Given a probability space $(X,\mathcal{B},\mu)$, we can construct a metric structure as follows: let $A$ be the result of quietening $\mathcal{B}$ by all the sets of measure zero, and we take the distance $d$ to be the measure of the symmetric difference between two sets. We can also add functions for the Boolean operations $\land$ (meet), $\lor$ (join), and complement. As distinguished elements we can name $0$ and $1$ in $A$. 
\item Metric structures also extend all classical first-order structures. Given a first-order structure $M$, we can turn it into a metric structure by considering the discrete metric ($d(a,b)=1$ if $a\neq b$ and zero otherwise). Any function is uniformly continuous with respect to this metric, and we capture the relations on $M$ by thinking about them as maps $R:M^n\to [0,1]$. 
\end{enumerate}
\end{example}

The notions of substructure, extension, reduct, etc., are defined in the natural way. We now define the language of continuous infinitary logic:

\begin{definition}\label{the-language}
Suppose that $\tau$ is a continuous vocabulary, we define the language $L^\mathbb{R}_{\omega_1,\omega}(\tau)$ recursively as follows. We start by fixing a set $\{x_i\mid i \in\omega\}$ of distinct variable symbols, then:
\begin{enumerate}
\item each $x_i$ is a $\tau$-term that respects the $\mathbb{N}$-modulus $\Delta_{x_i}(r_0,r_1,\dotsc)=r_i$ (we could instead ask that it respects the $n$-ary modulus $\Delta_{x_i,n}(r_0,r_1,\dotsc,r_n)=r_i$  for every $n\geq i$,);
\item if $t_0,\dotsc, t_{n-1}$ are terms and $f\in\tau$ is a function symbol with $n_f=n$, then $f(t_0,\dotsc, t_{n-1})$ is a $\tau$-term that respects the modulus $\Delta_f\circ (\Delta_{t_i}\mid i<n)$;

\item a $\tau$-atomic formula is an expression of the form $P(t_1,\dotsc, t_{n_P})$ where $P\in \tau$ is a predicate symbol and $t_1,\dotsc, t_{n_P}$ are $\tau$-terms. A $\tau$-atomic formula respects the bound $I_p$ and the modulus $\Delta_P\circ (\Delta_{t_i}\mid i<n)$;

\item the zero-ary functions $\bzero$ and $\bone$ are $\tau$-formulas with bounds $I_\bzero=\{0\}$ and $I_\bone=\{1\}$;

\item If $\phi$ and $\psi$ are $\tau$-formulas, then so are $\phi+\psi$,  $\max(\phi,\psi)$, $\min(\phi,\psi)$, and $q\phi$ for all $q\in\mathbb{Q}$. The corresponding bounds and moduli are easily computable from those of $\phi$ and $\psi$. These are the only connectives we allow;

\item If $\phi(x_0,\dotsc, x_{n-1})$ is a $\tau$-formula, then $\sup_{x_i} \phi(x_0,\dotsc, x_{n-1})$ and $\inf_{x_i} \phi(x_0,\dotsc, x_{n-1})$  are also $\tau$-formulas for $ i<n$. These formulas respect $I_\phi$ and the $(n-1)$-ary modulus
$$\hat\Delta_\phi(r_0,\dotsc,r_{n-2})=\Delta_\phi(r_0,\dotsc,r_{i-1},0,r_{i+1},\dotsc,r_{n-2});$$

\item If $\{\phi_n\mid n\in\omega\}$ is a countable set of $\tau$-formulas all of whose free variables are contained in a finite set of variables $\overline x$, and there is a modulus $\Delta$ and a compact set $I\subset \mathbb{R}$ such that every $\phi_i$ respects both $\Delta$ and $I$. Then $\sup_n \phi_n$ and $\inf_n\phi_n$ are $\tau$-formulas with free variables $\overline x$ respecting $\Delta$ and $I$;
\end{enumerate}
Then $L^\mathbb{R}_{\omega_1,\omega}(\tau)$ is the set of all $\tau$-formulas. A $\tau$ formula is quantifier free if it is built using only steps (1)-(5). Similarly, a formula is finitary or $L^\mathbb{R}_{\omega,\omega}(\tau)$  if it is built using only steps (1)-(6). We can also give the natural definitions of free variable and sentence. There is also a natural notion of equality for formulas: 
\end{definition}

\begin{definition}
We say that two $\tau$-formulas $\phi(\bar x)$ and $\psi(\bar x)$ are logically equivalent, and write $\phi(\bar x)=\psi(\bar x)$, if $\phi^\tA(\bar a)=\psi^\tA(\bar a)$ for every $\tau$-structure $\tA$ and every $\bar a\in \tA^{|\bar x|}$.
\end{definition}

\begin{remark}
\begin{enumerate}
\item In $L^\mathbb{R}_{\omega_1,\omega}(\tau)$ , $\mathbb{R}$ refers to the fact that we are in a continuous setting, $\omega$ is a bound on the number of free variables in any formula (as in every formula has less than $\omega$ many free variables), and $\omega_1$ is a strict bound on the number of formulas we can $\sup$ or $\inf$ at once.
\item Note the additional requirements  in point (7) of \Cref{the-language}. In particular, we can only take a countable supremum or infimum of a set of formulas if all the formulas satisfy the same modulus of continuity. We will deal with this requirement in the next subsection. 
\item The lattice version of the Stone-Weierstrass Theorem shows that our connectives are dense in the set of all continuous connectives as defined in \cite{benyaacov2017}. Moreover, as we have access to countable supremum and infimum, we also have access to $\limsup$ and $\liminf$, so any  continuous connective is expressible in our language using two infinitary qualifiers. 
\item We will often use the connectives
 $$\phi\ \dot-\ \psi = \max(\phi-\psi, 0)$$ 
 $$|\phi|=\max(\phi,0)+(-1)\min(\phi,0);$$
 \item We should also mention a peculiar fact about continuous infinitary logic mentioned in Remark 2.4 of \cite{hallback}: continuous infinitary logic is not closed under substituting constants for variables.
\end{enumerate}
\end{remark}

Another feature of this language is that every formula is equivalent to one in prenex normal form, i.e. a formula obtained from a quantifier-free formula by only using quantifiers:

\begin{lemma}\label{prenex-rules}
Every $\tau$-formula is equivalent to a $\tau$-formula in prenex normal form.
\end{lemma}

\begin{proof}
Fix $\tau$-formulas $\phi(\overline x,y)$ and $\psi(\overline x)$ be $\tau$-formulas. Without loss of generality, assume that $y$ does not occur in $\psi$. Then
\begin{align*}
(\sup_y\phi) + \psi &= \sup_y (\phi+\psi)\\
\max(\sup_y \phi, \psi) &= \sup_y \max(\phi, \psi)\\
\min(\sup_y \phi, \psi) &= \sup_y \min(\phi, \psi)
\end{align*}
Multiplication by rationals is a bit more complicated. For $q\in\mathbb{Q}$:
$$q(\sup_y\phi) = \begin{cases} 
	\sup_y q\phi &\text{if $q>0$}\\
	\inf_y q\phi &\text{if $q<0$}\\
	 \bzero  &\text{if $q=0$}\\
\end{cases}$$
The situation for $\inf$ is analogous. Thus, by induction on the complexity of formulas, we can always write an equivalent formula in prenex normal form. This completes the proof. 
\end{proof}

As in \cite{marker_lectures_2016}, we use $\sim\phi$ as a shorthand for applying the continuous and infinitary versions of DeMorgan's Laws to a formula in prenex normal form. That is, moving a factor of $-1$ through all the quantifiers, turning $\sup$ into $\inf$ and vise-versa in the process. 

\begin{definition}[Truth in a Model]
An easy induction argument shows that if $\tA$ is a metric $\tau$-structure, and $\phi(x_1,\dotsc, x_n)$ is a $\tau$-formula, then $\phi$ is interpreted in $\tA$ as a uniformly continuous function $\phi^\tA:A^n\to \mathbb{R}$. If $\phi$ is a sentence (i.e. it has no free variables), we consider the value of $\phi^\tA\in \mathbb{R}$ as a measure of how true (or really false) $\phi$ is in $\tA$. We will follow the convention that $0$ (or more generally $\min I_\phi$) corresponds to true.
\end{definition}

Finally, we note the following result from \cite{benyaacov2017}:

\begin{lemma}[\cite{benyaacov2017}  Lemma 2.4]\label{finitary-density}\phantom{}\\
\begin{enumerate}
\item We can define a seminorm on the set of formulas given by
$$\lVert\phi(\bx)\rVert = \sup\{|\phi^\tA(\bar a)|\mid \tA\ \text{is a metric $\tau$-structure and $\bar a\in A^{|\bx|}$}\}$$
and $\lVert\phi(\bx)\rVert \in I_\phi$;
\item If $|\tau|\leq \kappa$, for some infinite cardinal $\kappa$, then the space of $L^\mathbb{R}_{\omega,\omega}(\tau)$ formulas equipped with this norm has density character $\leq\kappa$.
\end{enumerate}
\end{lemma}

\smallskip

\subsection{Weak Modulus of Continuity} First introduced in \cite{benyaacov2017}, weak moduli of continuity allow us to specify a large and rich enough subset of formulas from the language where we may take countable suprema and infima freely without worrying about the continuity, without losing expressive power.

\begin{definition}
A \textit{weak modulus of continuity} is a function $\Omega:[0,\infty)^\omega\to[0,\infty]$ that satisfies: $\Omega(\overline 0)=0$; $$\Omega(\overline r)\leq\Omega(\overline r+\overline s)\leq\Omega(\overline r)+\Delta(\overline s)$$
for all $\overline r,\overline s\in[0,\infty)^\omega$; $\Omega$ is lower-semicontinuous in the product topology (i.e. the premiages of rays $(y,\infty]$ are open), and continuous in each argument.\par
 We say that a $\tau$-formula $\phi$ of arity $n$ respects $\Omega$ if it respects $\Omega|_n$ as a modulus of continuity. 
 \end{definition}
 
 Fixing a weak modulus of continuity $\Omega$ is equivalent to fixing a coherent sequence of moduli $\Omega|_n$ for each arity:
 
 \begin{lemma}[\cite{benyaacov2017} Lemma 2.3]
 Let $\Omega:[0,\infty)^\omega\to[0,\infty]$ be a weak modulus of continuity, then, for every $n\geq1$, the truncations $\Omega|_n$ is an $n$-ary modulus of continuity and $\Omega(x_0,x_1\dotsc)=\sup_n\Omega|_n(x_0,\dotsc,x_{n-1})$.
 \end{lemma}
 
 As in Section 5 of  \cite{benyaacov2017}, our focus in the later sections will be on universal weak-moduli: a weak-modulus such that every quantifier-free formula satisfies $\Omega$ after adding some number of initial dummy variables.

\begin{definition}[c.f. \cite{benyaacov2017} Definition 5.1]\label{Udef}
Suppose $\tau$ is a countable continuous vocabulary. A weak modulus $\Omega$ is universal for $L=L^\mathbb{R}_{\omega_1,\omega}(\tau)$ if it satisfies the following conditions:
\begin{enumerate}[label=(\roman*)]
\item For every atomic formula $\phi(x_0,\dotsc , x_{k-1})$, there is an $N<\omega$ such that
$$\Delta_\phi(r_0,\dotsc, r_{k-1})\leq \Omega|_N(0,\dotsc, 0, r_0,\dotsc, r_{k-1});$$
\item For every $k<\omega$ and every $M > 0$, there is an $N<\omega$ such that
$$M\cdot \Omega|_k (r_0,\dotsc,r_{k-1}) \leq \Omega|_N(0,\dotsc, 0, r_0, \dotsc, r_{k-1});$$
\item For every increasing sequence $i_0<i_1<\dotsc$, and every $\overline r\in [0,\infty)^\omega$, let $\overline s(i_k)=\overline r(k)$ and $\overline s(k)=0$ otherwise. Then $\Omega(\overline r)\leq\Omega(\overline s);$
\item For every $k,n<\omega$,
$$ \Omega|_k(r_0, \dotsc, r_{k-1}) +  \Omega|_n(s_0, \dotsc, s_{n-1}) \leq \Omega|_{k+n}(r_0,\dotsc, r_{k-1}, s_0, \dotsc, s_{n-1}).$$
\item For every $k<\omega$, let $d_k(x_0,\dotsc, x_{2k-1}) =  \max_{i<k} d(x_{i},x_{k+i})$. Then $\Delta_{d_k}\leq \Omega|_{2k}$.

\end{enumerate}
\end{definition}

\begin{lemma}[c.f. \cite{benyaacov2017} Propositions 5.2]\label{Upower}
Let $\Omega$ be a universal weak modulus for $L$, then
\begin{enumerate}[label=(\roman*)]
\item For every atomic formula $\phi(x_0,\dotsc , x_{k-1})$, there is an $N<\omega$ such that
$\psi(x_0,\dotsc, x_{N-1})=\phi(x_{N-k}, \dotsc, x_{N-1})$ respects $\Omega$;
\item For every formula $\phi(x_0,\dotsc , x_{k-1})$ that respects $\Omega$ and every $M > 0$, there is an $N<\omega$ such that $\psi(x_0,\dotsc, x_{N-1})=M\phi(x_{N-k}, \dotsc, x_{N-1})$ respects $\Omega$;
\item For every $M>0$ and every atomic formula $\phi(x_0,\dotsc , x_{k-1})$, there is an $N=N(M,\phi)$ such that for every increasing sequence $\sigma\in\mathbb{N}^k$ with $\min\sigma>N$, $\psi(x_0,\dotsc, x_{\max\sigma})=M\cdot\phi(x_{\sigma(0)},\dotsc,x_{\sigma(k-1)})$ is a quantifier-free $\tau$-formula that respects $\Omega$;
\item For all tuples $i_0<\dotsc < i_{n-1}< j_0 <\dotsc< j_{n-1}$, the formula
$$d_{\Omega|_n}((x_{i_0},\dotsc, x_{i_{n-1}}), (x_{ j_0},\dotsc, x_{j_{n-1}})) = \Omega(d(x_{i_0},x_{j_0}),\dotsc, d(x_{i_{n-1}},x_{j_{n-1}}))$$
respects $\Omega$. Moreover,
$$\max_{k<n} d(x_{i_k},x_{j_k})  \leq d^{\Omega|_n}((x_{i_0},\dotsc, x_{i_{n-1}}), (x_{ j_0},\dotsc, x_{j_{n-1}}))$$
\item $\Omega|_k$ is faithful for all $k<\omega$. 
\end{enumerate}
\end{lemma}

Our definition of a universal weak modulus of continuity is largely the same as in \cite{benyaacov2017}. We add $\textit{(v)}$ to \Cref{Udef}, which is compatible with the original definition, and allows for some compatibility with the sup metric.

 \begin{lemma}[c.f. \cite{benyaacov2017} Proposition 5.3]
Suppose $\tau$ is a countable continuous vocabulary. Then there exists a weak modulus $\Omega$ that is universal for $L=L^\mathbb{R}_{\omega_1,\omega}(\tau)$.
 \end{lemma}
  
\begin{proof}
Fix an enumeration $\{\phi_{2k+1}\}_{k\in\omega}$ of all atomic formulas in $L$, and let
$$\phi_{2k} = d_k(x_0,\dotsc, x_{2k-1}) =  \max_{i<k} d(x_{i},x_{k+i}).$$ 
Then, with $\Delta_{\phi_0}=id$, we write 
$$\Omega_U(L)(r_0,r_1,r_2,\dotsc) = \sum_{i=0}^\infty (i+1)\cdot \sup_{k\leq i} \Delta_{\phi_k}(r_i,\dotsc, r_i)$$
is a universal modulus. Note that $\Omega$ has range in $[0,\infty]$, so the infinite sum might not converge. 
\end{proof}

\begin{remark} There is one important caveat not present in \cite{benyaacov2017}: since we put a restriction on the connectives we are allowed to use, it is possible that the $d_{\Omega|_n}$ are no longer quantifier-free, but instead are the uniform limit of quantifier-free formulas. However, as these formulas will be our distances for the Cartesian powers, we will treat them as quantifier free formulas. One way to do this formally is to expand the language by adding a predicate for each of the distances $d_{\Omega|_n}$. Each $\tau$-structure has a unique expansion to this new language and $\Omega$ is a universal weak-modulus for the new language. 
\end{remark}

\subsection{Quantifier complexity}

 \begin{definition}
 Let  $\phi(x_1,\dotsc, x_n)$ be a $\tau$-formula, $\Omega$ a weak-modulus of continuity, and $I\subset\mathbb{R}$ compact. We say that $\phi$ is a weak-$(\Omega,I)$-formula if it respects $\Omega|_n$ as a uniform modulus of continuity and its image is contained in $I$.
 \end{definition}
 
This definition is a less restrictive class than that defined in Section 2.2 of  \cite{benyaacov2017}. However, we believe it is the correct notion to analyze definability.\par

Having a notion of prenex normal form allows us to define the quantifier complexity of a formula as follows:

\begin{definition}
Let $\alpha$ be an ordinal. We recursively define the quantifier complexity of  a $\tau$-formula $\phi$ as follows:
\begin{enumerate}
\item If $\phi(\overline x)$ is a quantifier-free $\tau$-formula, then $\phi(\overline x)$ is a $\sup^0=\inf^0$ formula.
\item $\phi(\overline x)$ is an $\inf^\alpha$ formula if it is logically equivalent to one of the form $\inf_n \inf_{\overline y_n} \psi_n(\overline x,\overline y_n)$ where each $\psi_n(\overline x,\overline y_n)$ is $\sup^\beta$ formula for some $\beta<\alpha$.
\item $\phi(\overline x)$ is a $\sup^\alpha$ formula if is logically equivalent to one of the form $\sup_n\sup_{\overline y_n} \psi_n(\overline x,\overline y_n)$ where each $\psi_n(\overline x,\overline y_n)$ is $\inf^\beta$ formula for some $\beta<\alpha$.
\end{enumerate}
\end{definition}

By induction on complexity, we get that the quantifier complexity of every $\tau$-formula is a countable ordinal. 

\begin{lemma}\label{ok-switch}
For any weak-$(\Omega,I)$-formulas $\phi_n(\bx, \by)$ we have
$$\psi(\bx) =  \inf_{\by} \inf_n \phi_n(\bx, \by) = \inf_n \inf_{\by} \phi_n(\bx, \by).$$
(i.e. we can change the order of the quantifiers up to logical equivalence). Moreover, 
\begin{enumerate}
\item the formulas $\phi^{\by}_n(\bx) = \inf_{\by} \phi_n(\bx, \by)$ are weak $(\Omega,I)$-formulas, and thus so is $\psi(\bx)$;
\item if $\phi_n(\bx,\by)$ is $\inf^\alpha$, for $\alpha>0$, then so is $\phi^{\by}_n(\bx)$. In particular, if every $\phi_n(\bx,\by)$ is $\inf^\alpha$,  for $\alpha>0$, then so is $\psi(\bx)$.
\end{enumerate}
A similar result holds when $\inf$ is replaced for $\sup$ everywhere.
\end{lemma}

\begin{proof}
Fix a metric $\tau$-structure $\tA$,  $\bar a\in \tA^{|\bx|}$, and some real $t\in I$. Suppose that 
$$ (\inf_{\by} \inf_n \phi_n(\bar a, \by))^{\tA} < t.$$ Then, there is some $\bar b\in  \tA^{|\by|}$ and some $n<\omega$ such that $\phi_n(\bar a,\bar b)^{\tA}<t$. Therefore,
$$  (\inf_n \inf_{\by} \phi_n(\bar a, \by))^{\tA}  \leq \phi_n(\bar a,\bar b)^{\tA}<t.$$
The other direction follows with the same proof.\par
The first moreover statement follows from the rules for constructing formulas, and the fact that the countable infimum of weak-$(\Omega,I)$-formulas is a weak-$(\Omega,I)$-formula. The second moreover statement follows from the definition of $\inf^\alpha$ formulas.
\end{proof}

\subsection{Definable Predicates}

\begin{definition} 
Let $\tA$ be a metric $\tau$-structure. A \textit{definable predicate} for $\tA$ is a function $P:A^n\to I$, for some closed interval $I\subset\mathbb{R}$, such that $P$ is the uniform limit of $\tau$-formulas. We say that $P$ is definable over $B\subset A$ if $P$ is the uniform limit of $\tau$-formulas with parameters from $B$. Furthermore, we say that $P$ is a weakly-$(\Omega,I)$ if it satisfies $\Omega$ as a weak modulus of continuity, and its image is contained in $I$.
\end{definition}

We want to extend our notions of complexity to definable predicates. Doing so necessitates the following lemma:

\begin{lemma}\label{good-approx}
Fix a weak modulus of continuity $\Omega$, and a closed interval $I=[I_{\min},I_{\max}]\subset \mathbb{R}$. Let $P:A^n\to I$ be a weakly-$(\Omega,I)$ definable predicate, and $\alpha>0$. Further suppose that $P$ is the uniform limit of $\inf^\alpha$-$\tau$-formulas $\{\phi_n(\bar x)\}_{n<\omega}$. Then $P$ is the uniform limit of weak-$(\Omega,I)$-$\inf^\alpha$-$\tau$-formulas.\par
A similar result holds when $\inf$ is replaced for $\sup$ everywhere.
\end{lemma}

\begin{proof}
Write
$$\phi_n^{I}(\bx) =  \max(I_{\min}, \min(I_{\max}, \phi_n(\bx))) $$
and
$$\phi_n^{(\Omega,I)}(\bx) =  \inf_{\bar y} \left(\phi_n^I(\bar y) + d^\Omega(\bx, \bar y) \right)$$
It is clear from construction that $\phi_n^{(\Omega,I)}$ has its image contained in $I$, and it is an $\inf^\alpha$ formula by \Cref{prenex-rules} and \Cref{ok-switch}.\par

We now show that $\phi_n^{(\Omega,I)}(\bx)$  satisfies $\Omega$. Indeed, fix $\bar a,\bar b\in A^{\bx}$, and assume without loss of generality that $\phi_n^{(\Omega,I)}(\bar a)^\tA\geq \phi_n^{(\Omega,I)}(\bar b)^\tA$. Thus, it is enough to show that 
$$ \phi_n^{(\Omega,I)}(\bar a)^\tA - \phi_n^{(\Omega,I)}(\bar b)^\tA \leq d^\Omega(\bar a,\bar b)^\tA.$$
Fix $\epsilon>0$, and  $\bar z\in A^{\bx}$ such that 
$$\phi_n^{(\Omega,I)}(\bar b)^\tA  \geq \phi_n^I(\bar z)^\tA + d^\Omega(\bar b,\bar z)^\tA -\epsilon$$
which is possible by the definition of $ \phi_n^{(\Omega,I)}$. Note that we also have
$$\phi_n^{(\Omega,I)}(\bar a)^\tA  \leq \phi_n^I(\bar z)^\tA + d^\Omega(\bar a,\bar z)^\tA $$
by definition. Now
\begin{align*}
\phi_n^{(\Omega,I)}(\bar a)^\tA - \phi_n^{(\Omega,I)}(\bar b)^\tA &\leq \phi_n^{(\Omega,I)}(\bar a)^\tA - \phi_n^I(\bar z)^\tA - d^\Omega(\bar b,\bar z)^\tA + \epsilon\\
&\leq \phi_n^I(\bar z)^\tA  + d^\Omega(\bar a,\bar z)^\tA - \phi_n^I(\bar z)^\tA - d^\Omega(\bar b,\bar z)^\tA + \epsilon\\
&=  d^\Omega(\bar a,\bar z)^\tA - d^\Omega(\bar b,\bar z)^\tA + \epsilon\\
&\leq  d^\Omega(\bar a,\bar b)^\tA +  d^\Omega(\bar z,\bar b)^\tA - d^\Omega(\bar b,\bar z)^\tA + \epsilon\\
&\leq  d^\Omega(\bar a,\bar b)^\tA+ \epsilon
\end{align*}
and taking $\epsilon\to 0$ gives the desired result.\par
It remains to show that $\phi_n^{(\Omega,I)}(\bar x)\to P(\bar x)$ uniformly. Fix $\epsilon>0$, and $n<\omega$ such that $\norm{P-\phi_n}_\infty<\epsilon$. Note that since the image of $P$ is contained in $I$, and $\phi_n$ approaches $P$ everywhere, we also get that $\norm{P-\phi^I_n}_\infty<\epsilon$. The definition of $\phi_n^{(\Omega,I)}$ readily gives us that 
$$\phi_n^{(\Omega,I)}(\bar a)^\tA - P(\bar a)^\tA \leq \phi_n^I(\bar a)^\tA-P(\bar a)^\tA<\epsilon$$
for any $\bar a\in A^{|\bx|}$.\par
For the other inequality, we use the same idea as in the previous part of this proof. Take  $\bar z\in A^{\bx}$ such that
$$\phi_n^{(\Omega,I)}(\bar a)^\tA  \geq \phi_n^I(\bar z)^\tA + d^\Omega(\bar a,\bar z)^\tA -\epsilon.$$
Then, using the fact that $P$ respects $\Omega$ as a modulus of continuity we can reason that
\begin{align*}
P(\bar a)^\tA - \phi_n^{(\Omega,I)}(\bar a)^\tA &\leq P(\bar a)^\tA - \phi_n^I(\bar z)^\tA - d^\Omega(\bar a,\bar z)^\tA +\epsilon\\
&\leq P(\bar z)^\tA + d^\Omega(\bar a,\bar z)^\tA - \phi_n^I(\bar z)^\tA - d^\Omega(\bar a,\bar z)^\tA +\epsilon\\
&= P(\bar z)^\tA  - \phi_n^I(\bar z)^\tA +\epsilon\\
&\leq 2\epsilon
\end{align*}
This completes the proof. For the $\sup$ case we use
$$\phi_n^{(\Omega,I)}(\bx) = \sup_{\bar y} \left(\phi_n^I(\bar y) - d^\Omega(\bx, \bar y) \right)$$
and proceed in a similar manner.
\end{proof}

 The case $\alpha=0$ is a bit more subtle: if a weakly-$(\Omega,I)$ definable predicate $P$ is the uniform limit of quantifier-free $\tau$-formulas, then it is both the uniform limit of weak-$(\Omega,I)$-$\inf^1$-$\tau$-formulas and the uniform limit of weak-$(\Omega,I)$-$\sup^1$-$\tau$-formulas. Thus, we could see these predicates as the continuous analogue of $\Delta^0_1$. 
 
\begin{definition}
For $\alpha>0$, we say that a definable predicate $P$ is weakly-$(\Omega,I)$-$\inf^\alpha$ if it is the uniform limit of weak-$(\Omega,I)$-$\inf^\alpha$-$\tau$-formulas. Similarly with $\inf$ is replaced for  $\sup$ everywhere.
\end{definition}

\subsection{Definable Sets}

Starting in Section 4, we will be exploring the quantifier complexity of  formulas that define automorphism orbits inside a metric structure. Hence, here we recall the definition of definability in continuous logic, as it is a bit more subtle than what one would expect:

\begin{definition}[\cite{benyaacov2008}]
Let $\tA$ be a metric $\tau$-structure, $D$ a metrically closed subset of $A^n$, and $\partial(\bx,\bar y)$ a $\tau$-formula that defines a complete, bounded distance on $A^n$. We say that $D$ is $\partial$-definable with parameters from $B\subset A$ if the function 
$$\partial^\tA(\overline x,D): \bx\mapsto \inf_{\bar a\in D} \partial(\bx,\bar a)$$
 is an $L(B)$-definable predicate. That is, there are $L(B)$-formulas $\phi_n(\overline x)$ such that $\phi^\tA_n(\overline x) \to \partial^\tA(\overline x, X)$ uniformly. Furthermore, for $\alpha$ an ordinal, we say that $D$ is $(\Omega,\inf^\alpha)$-definable if the function $d^\Omega(\bx,D)$ is the uniform limit of weak-$(\Omega,\inf^\alpha)$-$\tau$-formulas. 
\end{definition}

This notion of definability is stronger than asking that $X$ be the zero-set of some formula, but it has a number of natural advantages. We refer the reader to Chapter 9 of \cite{benyaacov2008} for the details, but we record the following results as we will use them later:

\begin{lemma}[Proposition 2.10 of \cite{benyaacov2008}]\label{Henson-alpha}
Let $F,G:X\to [0,1]$ be arbitrary functions such that
$$\forall \epsilon>0\ \exists \delta>0\ \forall x\in X\ \left(F(x)\leq \delta\Rightarrow G(x)\leq\epsilon\right).$$
Then, there exists a continuous, increasing function $\alpha:[0,1]\to[0,1]$ such that $\alpha(0)=0$ and
$$\forall x\in X\left (G(x)\leq \alpha(F(X))\right)$$
\end{lemma}

\begin{proof}
Let $g:[0,1]\to [0,1]$ be a (possibly discontinuous) function defined by
$$g(t) = \sup \{G(x)\mid F(x)\leq t\}$$
for $t\in [0,1]$. Note that $g$ is an increasing function and that $G(x)\leq g(F(x))$. Also, the assumption on $F$ and $G$ implies that $g(0)=0$ and that $g(t)$ converges to $0$ as $t\to 0$.\par
Now we construct an increasing, continuous function $\alpha:[0,1]\to[0,1]$ with $\alpha(0)=0$ and $g(t)\leq \alpha(t)$ as follows: let $(t_n)_{n\in\omega}$ be a decreasing sequence in $[0,1]$ with $t_0=1$ and $\lim_{n\to\infty}t_n=0$. Set $\alpha(0)=0$, $\alpha(1)=1$, and $\alpha(t_n)=g(t_{n-1})$. Finally, take $\alpha$ to be linear on each segment of the form $[t_n,t_{n-1}]$.\par
By construction, $\alpha$ is continuous, as $g$ is continuous at $0$. Moreover, $g(t)\leq \alpha(t)$: to see this, note that for all $t\in[0,1]$ there is an $n\geq 1$ such that $t_{n+1}\leq t\leq t_{n}$, so $\alpha(t)$ is a convex combination of $g(t_n)$ and $g(t_{n-1})$, with $g(t)\leq g(t_n)\leq g(t_{n-1})$. Thus, $g(t)\leq g(t_n)\leq \alpha(t)$.
\end{proof}

This lemma allows us to prove the main result regarding definability:

\begin{theorem}[Proposition 9.19 of \cite{benyaacov2008}]\label{define}
Let $\tA$ be a metric $\tau$-structure, $D\subset A^n$ be a closed, nonempty subset, and $\partial(\bx,\bar y)$ a $\tau$-formula that defines a complete, bounded distance on $A^n$. Then the following are equivalent:
\begin{enumerate}
\item $D$ is a $\partial$-definable set; 
\item There is a definable predicate $P(\bar x)$ that vanishes on $D$, and for every $\epsilon>0$ there is a $\delta>0$ such that if $P^\tA(\bar x)<\delta$, then $\partial^\tA(\bar x,D) \leq \epsilon$;
\item There is a sequence $\{\phi_m\mid m\geq 1\}$ of $\tau$-formulas and a sequence  $\{\delta_m\mid m\geq 1\}$ of positive reals such that for all $m\geq1$ and all $\bx\in A^n$:
$$\bx\in D\Rightarrow \phi_n^\tA(\bx)=0$$
$$ \phi_n^\tA(\bx)\leq \delta_m \Rightarrow \partial^\tA(\bx,D)\leq 2^{-m};$$
\item Quantification over $D$ is possible: for every $\tau$-formula $\psi(\overline z,\overline x)$ with $k=|\overline z|$ and $n=|\overline x|$, if the map
$$\psi(\overline z,\overline x): (A^k,\rho)\times (A^n,\partial)\to (\mathbb{R}, d_{\text{euclidean}})$$
is uniformly continuous for some $\tau$-formula $\rho(\overline z,\overline w)$ that defines a complete, bounded distance on $A^k$. Then, there is a  definable predicate $Q(\overline y)$ such that for all $\overline a\in A^k$
$$Q^{\tA}(\overline a) = \inf_{\overline b\in D} \psi^{\tA}(\overline a,\overline b)$$
and similarly for $\sup$.
\end{enumerate}
The same statement holds over parameters. 
\end{theorem}

\begin{proof} $(1\Rightarrow 3)$ As $F(\bar x)=\partial(\bar x,D)$ is a definable predicate, there are $\tau$-formulas $\{\psi_m\mid m\geq 1\}$ converging uniformly to $P$. We may assume that 
$$|F^\tA(\bar x)-\psi_m^\tA(\bar x)|\leq \frac{1}{3}\cdot 2^{-m}$$
for $\bar x\in A^n$. Note that if $\bar x\in D$, then $F^\tA(\bar x)=0$ and thus $\psi_m^\tA(\bar x)\leq \frac{1}{3}\cdot 2^{-m}$. Also, if $\psi_m^\tA(\bar x)\leq \frac{2}{3}\cdot 2^{-m}$, then
$$F^\tA(\bar x)\leq \psi_m^\tA+ |P^\tA(\bar x)-\psi_m^\tA(\bar x)| \leq \frac{2}{3}\cdot 2^{-m} +  \frac{1}{3}\cdot 2^{-m} = 2^{-m}.$$
Thus, the formulas $\phi_m(\bar x) = \psi_m(\bar x)\ \dot-\ \frac{2^{-m}}{3}$ satisfy (3) with $\delta_m= \frac{2}{3}\cdot 2^{-m}$.\\\par

$(3\Rightarrow 2)$ Let $P(\bar x)=\sum_{m=1}^\infty 2^{-m}\phi_m(\bar x)$, which is the uniform limit of the partial sums.\\\par

$(2\Rightarrow 1)$ Without loss of generality, both  $P(\bar x)$ and $\partial(\bar x,D)$ have range $[0,1]$. We use \Cref{Henson-alpha} with $F(\bar x)=P^\tA(\bar x)$ and $G(\bar x)=\partial^\tA(\bar x,D)$. Thus, we get a continuous, increasing function $\alpha:[0,1]\to[0,1]$ and $\partial^\tA(\bar x,D)\leq \alpha(P^\tA(\bar x))$.\par
Write
$$H(\bar x) = \inf_y \min(\alpha(P(\bar x))+\partial(\bar x,\bar y), 1)$$
and observe that $H(\bar x)$ is a definable predicate. Indeed, if $\{\phi_m\}_{m\in\omega}$ is a sequence converging uniformly to $P(\bar x)$, then the formulas
$$\psi_m(\bar x) = \inf_y \min(\alpha(\phi_m(\bar x))+\partial(\bar x,\bar y), 1)$$
converge uniformly to $H(\bar x)$. We now claim that $H^\tA(\bar x)=\partial^\tA(\bar x, D)$ for all $\bar x\in A^n$.\par
First, if $\bar y\in D$, then $P^\tA(\bar y)=0$ and
$$H^\tA(\bar x)\leq \min(\alpha(0)+\partial(\bar x,\bar y),1) = d^\tA(\bar x,\bar y)$$
and so $H^\tA(\bar x)\leq d^\tA(\bar x, D)$. Conversely, recall that $\partial^\tA(\bar x,D)\leq \alpha(P^\tA(\bar x))$, so we have
\begin{align*}
H(\bar x) &\geq  \inf_y \min(\partial(\bar x,D)+\partial(\bar x,\bar y), 1)\\
&\geq \min(d(\bar x,D),1)\\
&= \partial(\bar x,D)
\end{align*}
by our assumption on the range of the metric.\\\par

$(1\Rightarrow 4)$ Without loss of generality, the distances $\partial$ and $\rho$ have range in $[0,1]$ (i.e. the bound is 1). Fist, note that $\psi(\bar x,\bar y)$ obeys some modulus of continuity 
$$\Delta_\psi:[0,\infty)^2\to [0,\infty)$$
so that for every $\bar z\in A^k$ and every $\bar x,\bar y\in A^n$ we have
$$|\psi(\bar z,\bar x)-\psi(\bar z, \bar y)|\leq \Delta_\psi(\rho(\bar z,\bar z), \partial(\bar x,\bar y))=\Delta_\psi(0, \partial(\bar x,\bar y)).$$
Moreover, $\Delta_\psi(0,r)$ is an increasing function with range $[0,\Delta_\psi(0,1)]$. Therefore, $\Delta_\psi(0,r)$ is a uniformly continuous function, so it is the uniform limit of the basic connective. Thus,
$$ Q_d(\bar z) = \inf_{\bar x} \psi(\bar z,\bar x) + \Delta_\psi(0, \partial(\bar x, D))$$
is a definable predicate. We claim that this is the desired predicate, meaning that $Q_d=Q$.\par
To see this, recall that
$$\psi(\bar z,\bar y)\leq \psi(\bar z, \bar x) + \Delta_\psi(0, \partial(\bar x,\bar y))$$
taking the infimum over all $\bar y\in D$ on both sides, and using the fact that $\Delta_\psi(0,r)$ is continous and increasing, we get
$$Q(\bar z) \leq \psi(\bar z, \bar x) + \Delta_\psi(0, \partial(\bar x,D))$$
for all $\bar z\in A^k$ and $\bar x\in A^n$. Thus, we can take the infimum over all $\bar x\in A^n$ on the right-hand side and get that $Q\leq Q_d$. The other direction is clear since $Q_d$ is an infimum over a strictly larger collection.\\\par

$(4\Rightarrow 1)$ $\partial^\tA(\bar x, D) = \inf_{\bar a\in D} \partial^\tA(\bar x,\bar a)$. 
\end{proof}

Note the alternation of strict and non-strict inequalities in the second condition. Finally, we record the following characterization of definability communicated to us by Ita{\"i} Ben Yaacov:

\begin{lemma}
Let $P:A^n\to I$ be a definable predicate with $\min I=0$. Then $P$ defines the $\partial$-distance to a closed, non-empty, set $D\subset A^n$ (its zero-set) if and only if the following two conditions hold in $\tA$ (meaning that they evaluate to 0):
$$\sup_{\bar x,\bar y} \left(\left(P(\bx)\ \dot-\ P(\bar y)\right)\ \dot-\ \partial(\bx,\bar y)\right)$$
$$\sup_{\bx} \inf_{\bar y} \left(P(\bar y) + (\partial(\bx,\bar y)\ \dot-\ P(\bar x))\right).$$
The same statement is true when $P(\bar x)$ has parameters from $\tA$. 
\end{lemma}

\begin{proof}
The first conditions says that
$$\forall x\forall y\ P(\bar x)\leq P(\bar y)+\partial(\bar x,\bar y)$$
meaning that $P$ is 1-Lipschitz with respect to $\partial$. The second condition attempts to express that
$$\forall x\exists y\ P(\bar y)=0 \land \partial(\bar y,\bar x)\leq P(\bar x)$$
but can only do so approximately: it only requires that for every $\bar x\in A^n$ and every $\epsilon>0$ there is a $\bar y\in A^n$ such that $P(\bar y)<\epsilon$ and $\partial(\bar x,\bar y)\leq P(\bar y)+\epsilon$. Nonetheless, if $P(\bar x)=\partial(\bar x,D)$ for some nonempty set $D\subset A^n$, then both conditions are satisfied.\par
For the converse, suppose $P$ satisfies both conditions and let $D$ be its zero-set. Fix some $\bar a\in A^n$ and some $\epsilon>0$; we use the second condition inductively to define a sequence $\{\bar a_n\}_{n\in\omega}\subset A^n$ such that $\bar a_0=\bar a$, $P(\bar a_{n+1})<\epsilon/2^n$ and $\partial(\bar a_{n+1},\bar a_n)\leq P(\bar a_n)+\epsilon/2^n$.\par
Note that by construction $\partial(\bar a_{n+1},\bar a_n)\leq (3\epsilon)/2^n$ for $n\geq 1$, and so the sequence is Cauchy and has a limit $\bar b\in A^n$. By construction, $P(\bar b)=\lim P(\bar a_n)=0$, so $\bar b\in D$ and $D$ is nonempty. Moreover,
$$\partial(\bar a, D)\leq \partial(\bar a,\bar b)\leq P(\bar a)+\partial(\bar a,\bar a_1)+\partial(\bar a_1, \bar b) \leq P(\bar a)+4\epsilon.$$
Thus, we have that $\partial(\bar x,D)\leq P(\bar x)$, and the converse inequality follows from the first condition in the lemma. This complete the proof.
\end{proof}

\subsection{Types}\label{section-types}
We begin by summarizing the basic definitions and facts about types in infinitary continuous logic:

\begin{definition}
Let $\Omega$ be a weak modulus of continuity, $\tA$ a metric $\tau$-structure, and $\bar a\in \tA$. Then
$$(\Omega,\mathcal{L})-\operatorname{tp}(\bar a)= \{(\phi,r) \mid \text{$\phi$ satisfies $\Omega$ and } \phi^\tA(\bar a)\leq r\}$$
Note that, up to logical equivalence, $(\Omega,\mathcal{L})-\operatorname{tp}(\bar a)$ is equal to $\{\phi\mid \phi^\tA(\bar a)=0\}$ after doing some routine modifications to the formulas. However, this modifications might increase the number of quantifiers for formulas of low complexity, as we only included multiplication by rationals, and not reals, in our basic connectives. Therefore, we prefer the definition above to maintain uniformity with the definitions below.

We can, however, also consider the set of definable predicates realized by $\bar a$, which also proves to be quite useful:
$$(\Omega,\bar{\mathcal{L}})-\optp(\bar a)=\{P \mid \text{$P$ is a weakly $\Omega$ definable predicate and  $P^\tA(\bar a)=0$}\}.$$
Moreover, we can further restrict our types to only consider formulas with a particular complexity. In future sections, we will be particularly interested in $\sup^\alpha$ types, so we will give this definition explicitly, but other complexity classes may be used. Thus, 
\begin{definition}
Let $\Omega$ be a weak modulus of continuity, $\tA$ a metric $\tau$-structure, $\bar a\in \tA$, and $\alpha>0$ a countable ordinal. Then,
$$(\Omega,\sup\phantom{}^\alpha)-\optp(\bar a)=\{(\phi,r) \mid \text{$\phi$ is an $\Omega$-$\sup\phantom{}^\alpha$ formula and } \phi^\tA(\bar a)\leq r\}$$
$$(\Omega,\overline{\sup\phantom{}^\alpha})-\optp(\bar a)=\{P \mid \text{$P$ is a weakly $\Omega$-$\sup\phantom{}^\alpha$ def. predicate and $P^\tA(\bar a)=0$}\}.$$
\end{definition}
We may also refer to types abstractly as sets of formulas or predicates. We adopt the convention that letters like $p$ and $q$ refer to complete types, and $\Phi$ and $\Psi$ are used for partial types. In particular, we might consider a partial $(\Omega, \overline{\sup^\alpha})$-type $\Psi(\bar x)$.\par

Although we make no notational difference between types consisting of formulas and types consisting of predicates, we do define the following two operations:
$$\overline{\Psi(\bar x)} = \{ P\mid \text{$P$ definable predicate and the uniform limit of formulas in }\Psi(\bar x)\}$$
$$\interior{\Psi(\bar x)}= \{ (\phi, r) \mid \text{$\phi-r$ is a definable predicate in } \Psi(\bar x)\}.$$
\end{definition}

Complete types are well understood in the literature, and a more detailed treatment can be found in \cite{benyaacov2017} or \cite{hallback}. By a theory $T$, we mean a set of pairs $(\phi,r)$, where $\phi$ is a sentence and $r\in\mathbb{R}$, and we say that $\tA\models T$ if $\phi^{\tA}=r$ for all $(\phi,r)\in T$. Given a countable fragment $\mathcal{F}\subset L_{\omega_1,\omega}^{\mathbb{R}}(\tau)$ and an $\mathcal{F}$-theory $T$, the space of finitely consistent (or approximately realizable) types $\widehat{S}_{\bar x}(T)$
is a topometric space. That is, it is equipped with both a topology (the logic topology) and a metric; moreover, the metric topology refines the logic topology and the metric is lower-semi-continuous with respect to the topology. We denote the logic topology by $\ell$ and the basic logic open sets are of the form 
$$\llbracket\phi<r\rrbracket = \{p\in \widehat{S}_{\bar x}(T)\mid \phi^p<r\}$$
and the distance is given by 
$$\partial(p,q)\leq s \Longleftrightarrow \forall\phi\in\mathcal{F}\ \left[\inf_y \max\left( d(\bar x,\bar y)\ \dot-\ s, |\phi(\bar y) - \phi^p|\right)\right]=0.$$
It is not immediately clear that $\partial$ defines a distance, and \cite{hallback} gives a nicer definition based on the density of realizable types in $\widehat{S}_{\bar x}(T)$. Relevant for this paper is the omitting types theorem originally due to Eagle and restated as it appears in \cite{hallback}:

\begin{theorem}[\cite{eagle}] Let $\mathcal{F}$ be a countable fragment and let $T$ be an $\mathcal{F}$-theory. If $X_n$ is an $\ell$-meager and $\partial$-open set of realizable types for every $n$, then, there is a separable model $\tA\models T$ that omits all of the $X_n$.
\end{theorem}
In the next section, we will explore when partial types can be omitted in a separable model. \cite{farah_omitting_2017} have shown that this problem is essentially intractable, so we will focus on a specific family of types where partial results can be obtained.

\section{Model Existence \& Type Omitting Theorems}
In this section, we prove a continuous analogue of the Model Existence Theorem due to Makkai (see \cite{marker_lectures_2016} for a reference), which is a Henkin-style construction that works in the absence of a notion of formal proofs. First a few definitions:

\begin{definition}
Let $\tau$ be a continuous vocabulary, we define the language $L_{\infty,\omega}^\mathbb{R}(\tau)$ in a similar way as in \Cref{the-language}. The only change is in condition (7), which now reads:
\begin{enumerate}
\item[$(7)_\infty$] If $X$ is set of $\tau$-formulas all of whose free variables are contained in a finite set of variables $\overline x$, and there is a modulus $\Delta$ and a compact set $I\subset \mathbb{R}$ such that every $\phi\in X$ respects both $\Delta$ and $I$. Then $\sup_{\phi\in X} \phi$ and $\inf_{\phi\in X} \phi$ are $\tau$-formulas with free variables $\overline x$ respecting $\Delta$ and $I$.
\end{enumerate}
\end{definition}

\begin{remark}
Below, $\tau$ is a countable continuous vocabulary with a distinguished infinite set of constant symbols $C$. Take $L=L_{\infty,\omega}^\mathbb{R}(\tau)$. For every sentence $\phi\in L$, we assume that $I_\phi$ is a closed interval, and that $I_d=[0,1]$. In the definition below, $(\phi,r)\in \sigma$ is supposed to be interpreted as ``it is consistent that $\phi<r$''.
\end{remark}

\begin{definition}\label{conprop}
A consistency property $\Sigma$ is a collection of countable sets $\sigma\subset \bigcup_{\phi\in Sen(L)} \{\phi\}\times ((I_\phi\setminus\{\min I_\phi\})\cap\mathbb{Q})$ such that for every $\sigma\in\Sigma$:
\begin{enumerate}[label={(C\arabic*)}]
\item If $\mu\subset\sigma$, then $\mu\in\Sigma$;
\item If $(\phi,r)\in\sigma$, then there is $s<r$ and $\mu\in\Sigma$ such that $\sigma\cup\{(\phi,s)\}\subset\mu$;
\item If $(\phi,r),(-\phi, s)\in \sigma$, then $-s<r$;
\item If $(-\phi,r)\in \sigma$, then there is $\mu\in\Sigma$ such that $\sigma\cup\{(\sim\phi, r)\}\subset\mu$;

\item If $(\phi\ +\ \psi,r)\in\sigma$, then there is $\mu\in\Sigma$ and $s\in\bbQ$ such that $\sigma\cup\{(\phi,s),(\psi,r-s)\}\subset\mu$;

\item If $(q\phi,r)\in\sigma$ for some $q\in\bbQ$, then 
\begin{enumerate}[label={(C6\alph*)}]
\item If $q>0$, there is $\mu\in\Sigma$ such that $\sigma\cup\{(\phi,r/q)\}\subset\mu$;
\item If $q<0$, there is $\mu\in\Sigma$ such that $\sigma\cup\{(-\phi,r/|q|)\}\subset\mu$;
\item If $q=0$, then $r>0$ and for all  $s\in\bbQ^+$  there is $\mu\in\Sigma$ such that $\sigma\cup\{(0\phi,s)\}\subset\mu$;
\end{enumerate}

\item If $(\sup_n\phi_n,r)\in\sigma$, then for all $n<\omega$ there is $\mu\in \Sigma$ such that $\sigma\cup\{(\phi_n,r)\}\subset\mu$;
\item If $(\inf_n\phi_n,r)\in\sigma$, then there is $\mu\in \Sigma$ and an $n<\omega$ such that $\sigma\cup\{(\phi_n,r)\}\subset\mu$;
\item If $(\sup_{\overline x}\phi(\overline x),r)\in\sigma$, then for all $\overline c\in C^{|\overline x|}$ there is $\mu\in \Sigma$ such that $\sigma\cup\{(\phi(\overline c),r)\}\subset\mu$;
\item If $(\inf_{\overline x}\phi(\overline x),r)\in\sigma$, then there is $\mu\in \Sigma$ and $\overline c\in C^{|\overline x|}$ such that $\sigma\cup\{(\phi(\overline c),r)\}\subset\mu$;
\item Let $t_1,\dotsc,t_n$ be terms with no variables and $a_1,\dotsc, a_n\in C$. If 
$$(d(a_1,t_1),r_1),\dotsc, (d(a_n,t_n),r_n), (\phi(t_1,\dotsc,t_n),s) \in \sigma,$$
then there is $\mu\in \Sigma$ and a rational number $q\leq\Delta_\phi(r_1,\dotsc, r_n)+s$ such that $\sigma\cup\{(\phi(a_1\dotsc, a_n),q)\}\subset \mu;$

\item Let $t$ be a term with no variable and $a,b\in C$, then
	\begin{enumerate}[label={(C12\alph*)}]
		\item If $(d(a,b),r)\in \sigma$, then there is  $\mu\in \Sigma$ such that $\sigma\cup\{(d(b,a),r)\}\subset \mu$;
    		\item For all rational $r>0$, here is $\mu\in\Sigma$ and $c\in C$ such that $\sigma\cup\{(d(c,t),r)\}\subset\mu$.
    	\end{enumerate}
\end{enumerate}
\end{definition}

\begin{lemma}\label{prop-of-con}
Let $\Sigma$ be a consistency property with $\sigma\in\Sigma$ and $a,b,c\in C$:
\begin{enumerate}
\item For any rational $r>0$, there is $\mu\in \Sigma$ and a positive rational $q\leq r$ with $\sigma\cup\{(d(a,a),q)\}\subset\mu$;
\item If $(d(a,b),r),(d(b,c),s)\in\sigma$ then there is $\mu\in\Sigma$ and a positive rational $q\leq r+s$ with $\sigma\cup\{(d(a,c),q)\}\subset\mu$.
\end{enumerate}
\end{lemma}

\begin{proof}
We begin the proof by noting that $\Delta_{d(x,c)}=Id$ for any $c\in C$. Indeed, recall that $\Delta_d(r_1,r_2)=r_1+r_2$, and $c$ is a function of arity zero, so it has a modulus of arity zero $\Delta_c=0$. Hence, $\Delta_{d(x,c)} = \Delta_d \circ (\Delta_x, \Delta_c)$ has arity 1 and $\Delta_{d(x,c)}(r) = r$. With this in mind:
\begin{enumerate}
\item Fix $r>0$, note that $t=\textrm{`a'}$ is a term in language, so by (C12b) there is $\mu\in\Sigma$ and $c\in C$ such that $\sigma\cup\{(d(c,a),r/2)\}\subset\mu$. Then (C12a) guarantees a further extension that contains $\sigma\cup\{(d(c,a),r/2),(d(a,c),r/2)\}$. Finally, we apply (C11), and the fact that $\Delta_{d(x,a)}=Id$, to find a positive rational $q\leq r/2+r/2=r$ and an extension in $\Sigma$ that contains $\sigma\cup\{(d(a,a),q)\}$.
\item Fix $r,s>0$, then apply (C11) to $t=\textrm{`b'}$ and $\phi(x)=d(x,c)$.
\end{enumerate}
\end{proof}

\subsection{The Model Existence Theorem}
\begin{theorem}\label{model-construction}
If $\Sigma$ is a consistency property and $\sigma\in\Sigma$, then there is a separable $\tau$-structure $\tA$ such that $\phi^{\tA}<r $ for every $(\phi,r)\in\sigma$.
\end{theorem}

\begin{proof}
Let $F\subset L$ be the smallest set of $L_{\infty,\omega}^\mathbb{R}(\tau)$-sentences such that:
\begin{itemize}
\item $\phi\in F$ for all sentences $\phi$ that appear in $\sigma$;
\item $F$ is closed under sub-sentences;
\item if $\phi(\overline v)$ is a sub-formula of a sentence in $F$ and $\overline c\in C$, then $\phi(\overline c)\in F$;
\item if $-\phi\in F$, then $\sim\phi\in F$;
\item if $a,b\in C$, then $d(a,b)\in F$.
\end{itemize}
Now fix a listing $\{\phi_i\}_{i\in\omega}$ of all $F$-sentences such that each sentence appears infinitely often. Similarly, fix a listing $\{t_i\}_{i\in\omega}$ of all $L$-terms without variables such that each term appears infinitely often, recall $\tau$ is countable. We now build a sequence 
$\sigma=\sigma_0\subset\sigma_1\subset\dotsc$ as follows: given $\sigma_i$ find $\sigma_{i+1}\in \Sigma$ with the following properties:

\begin{enumerate}[label={(P\arabic*)}]
\item there is some rational $r$ with $r<2^{-(i+1)}+\inf\{s\mid \exists\mu\in\Sigma,\ \sigma_i\cup\{(\phi_i,s)\}\subset\mu\}$ and $(\phi_i,r)\in\sigma_{i+1}$. Moreover, for $(\phi_i,r)$ we also require that: 
\begin{enumerate}[label={(\Alph*)}]
\item if $\phi_i$ is $\psi_1+\psi_2$, then $\sigma_{i+1}$ satisfies the conclusion of (C5);
\item if $\phi_i$ is $q\psi$ for some $q\in\mathbb{Q}$, then $\sigma_{i+1}$ satisfies the conclusion of (C6);
\item if $\phi_i$ is $\inf_n \psi_n$, then $\sigma_{i+1}$ satisfies the conclusion of (C8);
\item if $\phi_i$ is $\inf_{\overline v} \psi(\overline v)$, then $\sigma_{i+1}$ satisfies the conclusion of (C10);
\end{enumerate}
\item if $\phi_i$ is $\psi(a_1,\dotsc, a_n)$, for some $a_1,\dotsc,a_n\in C$, and there are terms $t_1,\dotsc, t_n$ with no variables such that
$$(d(a_1,t_1),r_1),\dotsc, (d(a_n,t_n),r_n), (\phi(t_1,\dotsc,t_n),s) \in \sigma_i,$$
then $\sigma_{i+1}$ satisfies the conclusion of (C11);
\item $(d(c,t_i),r)\in \sigma_{i+1}$ for some $c\in C$ and some $r\leq 2^{-(i+1)}$.
\end{enumerate}

Where  $\max I_\phi$ is the infimum if the set is empty. Write $\Gamma=\bigcup_n \sigma_n$. \par

Our goal is to define a metric structure made from the constants where the value of any sentence $\phi$ in the expanded language is $\inf\{r\mid (\phi,r)\in\Gamma\}$. Hence, at each stage, lowering this infimum is giving us more information about the value of the sentence $\phi$. Seen this way, (P1) ensures that every time we consider a formula, we extend to a condition that provides close to the maximal amount of information. (P2) guarantees that formulas satisfy their prescribed modulus of continuity, and (P3) will make the constants dense in the resulting structure \par

 We now build a metric structure using the constants. For constants $a,b\in C$, we say that $a\sim b$ if $\Gamma\vdash d(a,b)=0$ (i.e. if $\inf\{r\mid (d(a,b),r)\in\Gamma\} = 0$). 

\begin{claim} $\sim$ is an equivalence relation on $C$.
\end{claim}

\begin{proof}
Note that for any $a\in C$, $d(a,a)$ is a formula in $F$, so say it appears in the construction as $\phi_i$. By the first statement of \Cref{prop-of-con}, some extension of $\sigma_i$ in $\Sigma$ contains $(d(a,a),r_1)$ for some $r_1\leq 2^{-(i+1)}$, and then (P1) above implies that $(d(a,a),r)\in\sigma_{i+1}$ for some $r<r_1+2^{-(i+1)}\leq 2^{-i}$. Thus, as $d(a,a)$ is listed infinitely often, it follows that $\inf\{r\mid (d(a,a),r)\in\Gamma\} = 0$, so $c\sim c$.\par 

If $\inf\{r\mid (d(a,b),r)\in\Gamma\} = 0$, then for any $\epsilon>0$ there is an $s$  and an $i$ such that $(d(a,b),s)\in \sigma_{i}$, $\phi_i$ is $d(b,a)$, and $s+2^{-(i+1)}<\epsilon$ (here we use the fact that $d(b,a)$ appear arbitrarily far in the listing of sentences). By condition (C12a) of \Cref{conprop} and (P1), it follows that $(d(b,a),s')\in\sigma_{i+1}$ for some $s'<s+2^{-(i+1)}<\epsilon$. Thus, $b\sim a$.\par

Transitivity is similar: fix $\epsilon>0$ and choose $r,s,i$ such that $(d(a,b),r),(d(b,c),s)\in \sigma_i$, $\phi_i$ is $d(a,c)$, and $r+s+2^{-(i+1)}<\epsilon$. Then, apply the second statement in \Cref{prop-of-con} and (P1) to obtain that there is an $s'<r+s+2^{-(i+1)}<\epsilon$ such that $(d(a,c),s')\in\sigma_{i+1}$. 
\end{proof}

Let $\A=C/\sim$, and write $[c]$ for the $\sim$-equivalence class of $c\in C$. We can now proceed to define the intepretations of the predicates on $\A$:

\begin{claim}
For $a,b\in C$, let $d^\A([a],[b])=\inf\{r\mid (d(a,b),r)\in\Gamma\}$. Then $d^\A$ defines a metric on $\A$ that repects $I_d=[0,1]$ and $\Delta_d(r_1,r_2)=r_1+r_2$. 
\end{claim}

\begin{proof}
Recall that $I_d=[0,1]$, so by construction $0\leq d(a,b)\leq 1$ for all $a,b\in C$, and hence the same holds for the equivalence classes. It should also be noted that (C2) implies that if $s=\inf\{r\mid (d(a,b),r)\in\Gamma\}$, then $(d(a,b),s)\not\in\Gamma$, and thus ``it is consistent that the distance is equal to $s$.''  \par

To show $d^\A$ is a metric, note that $d^\A([a],[b])=0$ iff $[a]=[b]$ as we have quotiented out by precisely that equivalence relation. Symmetry of $d^\A$ can be proven using a similar argument to the one used in the previous claim. This leaves us with the triangle inequality: suppose $a,b,c\in C$ and rational $r,s$ are such that $(d(a,b),r),(d(b,c),s)\in\sigma_i$, and $\phi_i$ is $d(a,c)$. Then (P2) implies us there is an  $q\leq \Delta_{d(x,c)}(r)+s=r+s$ such that  $(d(a,c),q)\in \sigma_{i+1}$. As this is true for any values of $r$ and $s$, we obtain that
$$\inf\{q\mid (d(a,c),q)\in\Gamma\}\leq \inf\{r\mid (d(a,b),r)\in\Gamma\}+\inf\{s\mid (d(b,c),s)\in\Gamma\}$$
and $d^\A([a],[c])\leq d^\A([a],[b])+d^\A([b],[c])$.\par

Moreover, we can apply the triangle inequality a couple times to show that $d^\A$ satisfies $\Delta_d$: fix $a_1,a_2,b_1,b_2\in C$ and assume that $d^\A([b_1],[b_2])\leq d^\A([a_1],[a_2])$. Then:
\begin{align*}
d^\A([a_1],[a_2]) &\leq d^\A([a_1],[b_1]) + d^\A([b_1],[a_2])\\
&\leq d^\A([a_1],[b_1]) + d^\A([b_1],[b_2]) + d^\A([b_2],[a_2])
\end{align*}
so by our assumption we get that
\begin{align*}
| d^\A([a_1],[a_2]) -  d^\A([b_1],[b_2])| & \leq d^\A([a_1],[b_1]) + d^\A([b_2],[a_2])\\
&= d^{\Delta_d}(([a_1],[a_2]), ([b_1],[b_2])).
 \end{align*}
Finally, note that this definition does not depend on the representatives chosen for each equivalence class: given $\epsilon>0$, and assuming for simplicity we can always choose nice multiples of $\epsilon$ in our construction, suppose that $(d(a',a),\epsilon),(d(b',b),\epsilon), (d(a,b),r)\in\sigma_i$ for $i$ such that $2^{-(i+1)}<\epsilon$. Next, find $i_0>i$ such that $\phi_{i_0}$ is $d(a',b)$ and apply (P2) with $t_0=a, t_1=b$ to get $(d(a',b),r+\epsilon)\in\sigma_{i_0+1}$. Note that we can permute the variables at the cost of at most $\epsilon$, so take $i_1>i_0$ with $\phi_{i_1}$ is $d(b',a')$ and $(d(b,a'),r+2\epsilon)\in\sigma_{i+1}$; apply (P2) with $t_0=b, t_1=a'$ to $b'$ to get $(d(b',a'),r+3\epsilon)\in\sigma_{i_1+1})$. The result follows by an additional permutation of variables, and  using $\epsilon/4$ everywhere.
\end{proof}

We can extend the same argument to show that we have suitable interpretations for all predicates in $\tau$ as follows: for $P\in \tau$, an $n$-ary relation symbol, and $c_1,\dotsc, c_n\in C$, we define 
$$P^\A([c_1],\dotsc, [c_n]) = \inf\{r\mid (P(c_1,\dotsc,c_n),r)\in\Gamma\}$$
where again if the set is empty then we take $\max I_P$ as the infimum. It should be noted that, as with the distance, if $s$ is the infimum on the right-hand side, then $(P(c_1,\dotsc, c_n), s)\not\in \Gamma$ by (C2).  The construction guarantees that $P^\A$ satisfies $I_p$, and (P2) implies $P^\A$ satisfies $\Delta_P$. Moreover, we can also use (P2) as in the previous claim to show that this definition does not depend on the choice of representatives. 

Next, we want to interpret the function symbols of $\tau$ in $\A$, but $\Gamma$ only allows us to do so approximately. Hence, it is convenient to at this point to pass to $\tA$, the completion of $\A$ with respect to the metric $d^\A$. We also extend the interpretations of the predicates $P\in\tau$ such that $P^{\tA}$ is a continuous function that satisfies $\Delta_P$.\par

For the next claim, we say that $\Gamma$ proves that $\lim_n t_n=t$, for $\tau$-terms $t,\{t_n\}_{n<\omega}$ without variables, if for every $\epsilon>0$ there is an $N$ such that if $n>N$, then there is a $\delta_n<\epsilon$ such that $(d(a_n, a),\delta_n)\in \Gamma$.

\begin{claim}\label{approx-functions}
Let $f$ be an $n$-ary function symbol in $\tau$ and $c_1,\dotsc, c_n\in C$. Then, there is a sequence $\{a_j\}_{j\in\omega}\subset C$ such that $\{[a_j]\}_{j\in\omega}\subset C$  is Cauchy in $\tA$ and $\Gamma$ proves that the sequence to converge to $f(c_1,\dotsc, c_n)$ . 
\end{claim}
\begin{proof}
Let $i_0<i_i<\dotsc$ be the infinitely many indices such that $t_{i_j}$ is $f(c_1,\dotsc, c_n)$. To ease notation, will henceforth just call the term $t$ . Using (P3) from the definition of $\sigma_{i+1}$, choose $a_j\in C$ be such that $(d(a_j, t), r_j)\in \Gamma$ for some $r_j\leq 2^{-(i_j+1)}$.\par

To show that $\{[a_j]\}_{j\in\omega}\subset C$  is Cauchy in $\tA$, fix $j<k<\omega$. Let $i>k$ be large enough so that $(d(a_j, t), r_j), (d(a_k, t), r_k)\in \sigma_i$, and $\phi_i$ is $d(a_j, t)$. By (C11), there is a $\mu\in\Sigma$ such that $\sigma_i\cup\{(d(a_j,a_k),q)\}\subset\mu$ for some $q\leq \Delta_d(r_j)+r_k= r_j+r_k\leq 2^{-j}$, and by (P2) we have that $(d(a_j,a_k),q)\in\sigma_{i+1}$. Therefore, $d^\A([a_j],[a_k])\leq 2^{-j}$, from which the result follows.\par
Now that we know the sequence is Cauchy, it is clear from the choice of $a_j$ that $\Gamma$ proves $a_j\to f(c_1,\dotsc, c_n)$ as $j\to\infty$. Moreover, a similar argument to the one above shows that if $\{b_j\}_{j\in\omega}$ is any other sequence satisfying the assumptions of the claim, then $\Gamma$ proves that $d(a_j,b_j)\to 0$ as $j\to\infty$, so this notion is well defined.
\end{proof}

Thus, for an $n$-ary function symbol $f$ and $c_1,\dotsc c_n\in C$, we define $f^{\tA}([c_1],\dotsc,[c_n])$ to be the limit in $\tA$ of some, equivalently every, sequence guaranteed by the claim. As it has been done throughout this proof, (P2) applied to $d(f(\overline x),f(\overline y))$ shows that $f^{\tA}$ satisfies $\Delta_f$, so we can extend this definition to all of $\tA$ in the unique way that makes $f^{\tA}$ to be continuous. Next, we use induction on the terms to extend \Cref{approx-functions} to all $\tau$ terms:

\begin{claim}\label{approx-terms}
Let $t(v_1,\dotsc, v_n)$ be a $\tau$-term and $c_1,\dotsc, c_n\in C$. Then, there is a sequence $\{a_j\}_{j\in\omega}\subset C$ such that $\{[a_j]\}_{j\in\omega}\subset C$  is Cauchy in $\tA$ and $\Gamma$ proves that the sequence to converge to $t(c_1,\dotsc,c_n)$ . 
\end{claim}

Therefore, as before, we define  $t^{\tA}([c_1],\dotsc,[c_n])$ to be the limit of any such Cauchy sequence. Then  $t^{\tA}$ satisfies $\Delta_t$ and we extend this definition to all of $\tA$ in the unique way that makes the term continuous. Thus, $\tA$ is a complete metric space where every $\tau$-symbol is interpreted as a uniformly continuous function that satisfies the modulus of continuity, and the bound (for predicates). Thus, $\tA$ is a $\tau$-structure. It only remains to show that:

\begin{claim}
If $(\phi,r)\in\Gamma$, then $\phi^{\tA}<r$. In fact, $\phi^{\tA}=\inf\{r\mid (\phi,r)\in\Gamma\}$.
\end{claim}
\begin{proof}
The proof is by induction on the complexity of $\phi$:
\begin{itemize}[leftmargin=*]
\item If $\phi$ is $d(t_1,t_2)$, then let $\{a_j\}_{j\in\omega}$ and  $\{b_j\}_{j\in\omega}$ be the sequence guaranteed by \Cref{approx-terms}. Suppose that$(\phi,r)\in\sigma_{i_0}$ for some $i_0$, then the fact that $\phi_i$ is $\phi$ for infinitely many $i>i_0$, and (C2) and (P1), imply that
$$s=\inf\{s'\mid (\phi,s)\in\Gamma\}<r,$$
so let $\epsilon<r-s$ and assume going forward that we can always find nice multiples of $\epsilon$ inside $\Gamma$. Now, take $i_1>i_0$ and a $j_1<\omega$ such that:
\begin{enumerate}[label={(\roman*)}]
\item $\phi_{i_1}$ is $d(a_{j_1},t_2)$;
\item  $(d(t_1,t_2),s+\epsilon)\in\sigma_{i_1}$;
\item $(d(a_{j_1},t_1),\epsilon)\in\sigma_{i_1}$;
\item $\Gamma$ proves that for $j>j_1$, $d(a_{j_1},a_j)<\epsilon$. 
\end{enumerate}
we can now apply (P2) to see that $(d(a_{j_1},t_2),s+2\epsilon)\in \sigma_{i_1+1}$. Now take $i_2>i_1$ and $j_2$ with $\phi_{i_2}$ is $d(a_{j_1},b_{j_2})$ and similar properties to the ones above to see that $(d(a_{j_1},b_{j_2}),s+3\epsilon)\in\sigma_{i_2+1}$. Sending $\epsilon\to 0$, we see that if $a_j\to a=t_1^{\tA}$ and $b_j\to b=t_2^{\tA}$, then $d^{\tA}(a,b)=s<r$ and we are done.
\item The case when $\phi$ is $P(t_1,\dotsc, t_n)$ is similar to the one above, with one application of (P2) for each term (with a different sentence each time).

\item If $\phi$ is $\psi_1+\psi_2$, then choose a stage $i$ such that $\phi_i$ is $\psi_1+\psi_2$. Note that $(\psi_1+\psi_2,r)\in\sigma_i$, so (P1) and (A) imply that there is an $s\in\bbQ$ such that $(\psi_1,s)$ and $(\psi_2,r-s)$ are in $\sigma_{i+1}$, and we use induction to get $\psi_1^{\tA}<s$ and $\psi_2^{\tA}<r-s$.

\item If $\phi$ is $q\psi$, assume that $\phi_i$ is $q\psi$. For $q>0$, (B) implies that $\sigma_{i+1}$ contains $(\psi,s)$ for some $s\leq r/q$. By induction, $\psi^{\tA}<r/q$. If $q<0$, (B) and induction imply that $-\psi^{\tA}<r/|q|$, so $q\psi^{\tA}<r$. If $q=0$, then (B) implies that $(q\psi)^{\tA}=0$. 
\item If $\phi$ is $\sup_n\psi_n$, assume that $(\phi,r)\in \sigma_i$. Then choose $i<i_0<i_1<\dotsc$ such that $\phi_{i_n}$ is $\psi_n$, and now use (C2), (C7) and (P1) to see that $(\psi_n, r)\in \Gamma$, actually for some value $\leq r$, and the result follows by induction.
\item If $\phi$ is $\inf_n\psi_n$, assume that $(\phi,r)\in \sigma_i$. Then (C) guarantees that $(\psi_n,r)\in\sigma_{i+1}$ for some $n$ and proceed by induction. 
\item If $\phi$ is $\sup_{\overline v}\psi(\overline v)$, assume that $(\phi,r)\in \sigma_i$. Then choose $i<i_0<i_1<\dotsc$ such that $\phi_{i_n}$ is $\psi(\overline c_{i_n})$ , for a listing of $C^{|\overline v|}$, and now use (C2), (C9) and (P1)  to see that $(\psi(\overline c_n),r)\in \Gamma$, really for some value $\leq r$, and the result follows by induction.
\item If $\phi$ is $\inf_{\overline v}\psi(\overline v)$. Then condition (D) guarantees that $(\psi(\overline c),r)\in\sigma_{i+1}$ for some $\overline c\in C^{|\overline v|}$ and proceed by induction. 
\item If $\phi$ is $-\psi$, assume that $(\phi,r)\in \sigma_i$. Note that $\psi$ is $\phi_ j$ for infinitely many $j$, so  (C2) and (C3) imply that
$$-r<\sup\{-r'\mid (-\psi,r')\in\Gamma\}\leq s=\inf\{s'\mid (\psi,s')\in\Gamma\}$$
and by induction of $\psi$ we get that $\psi^{\tA}=s$, so $\left(-\psi\right)^{\tA}=-\left(\psi^{\tA}\right) = -s<r$ and we are done.
\end{itemize}
\end{proof}
Thus, as $\sigma=\sigma_0\subset\Gamma$, $\tA$ is the desired model.
\end{proof}
\begin{remark}
Observe that it is sufficient to construct $\Gamma$ satisfying properties (P1)-(P3) to construct a model.
\end{remark}

\subsection{The Type Omitting Theorem} The main reason to prove this version of the model existence theorem is to prove a specific kind of type omitting theorem for infinitary continuous logic. Our objective is to find a generalization of the following theorem of Montalb\'an:

\begin{theorem}[\cite{montalban_robuster_2015}]
Let $\mathcal{A}$ be a countable structure, $\phi$ a $\Pi_{\alpha+1}$ sentence true in $\mathcal{A}$, and $\Phi(\bar x)$ be a partial $\Pi_\alpha$ type realized in $\mathcal{A}$ but not $\Sigma_{\alpha}$ supported in $\mathcal{A}$. Then, there is a countable structure $\mathcal{B}$ such that $\mathcal{B}\models\phi$ but omits $\Psi$. 
\end{theorem}

 Here, supported is a notion close to isolation, and enough to guarantee that we can omit $\Psi$ while still satisfying the sentence $\phi$.The proof of the theorem proceeds by constructing a set of formulas akin to $\Gamma$ in \Cref{model-construction} using $\Sigma_{\alpha}$ sentence in an expanded language with Henkin constants.\par

For the continuous case, following the ideas of Ben Yaacov and Eagle on type omitting, we want to define a notion of a type being supported so that if an $(\Omega,\sup^\alpha)$-type $\Psi$ is not supported, then some set containing it should be topologically meager and metrically open. Our main issue is that it is not clear that partial types live in a natural topometric space. In particular, the definition for distance given in \Cref{section-types} does not, at the moment, seem to work for, say, $(\Omega,\sup^1)$-types as the formulas on the right hand side are $\inf^2$. This increase in complexity is one of the main differences between the discrete case developed in \cite{montalban_robuster_2015} and \cite{montalban2021} and the continuous one presented here. It remains open whether the result on this section, and therefore this paper as a whole, can be strengthen to work for a broader collection of partial types.\par 

Nevertheless, there is a natural collection of types where we can use the distance defined in \Cref{section-types} to omit types: complete $(\Omega,\sup^{<\alpha})$-types for $\alpha$ is a countable limit ordinal:

\begin{definition}
Fix a weak modulus of continuity $\Omega$. Given an $\Omega$-formula $\psi(\bar x)$, a real number $r$, and some $\epsilon>0$, let 
$$\theta(\psi,r,\epsilon)(\bx) = \inf_{\by} \max(d_\Omega(\bx,\by)\dot- \epsilon\bone, |\psi(\by)-r\bone|).$$
\end{definition}

We have that $\theta(\psi,r,\epsilon)^{\tA}(\bx)=0$ if, for every $\delta>0$, there is some $\by$ within $\epsilon$ of $\bx$, as measured by $d_\Omega$, such that $|\psi^\tA(\by)-r|<\delta$. In particular, if $\theta(\bx)>0$, then no element in the $\epsilon$-ball centered at $\bx$ believes that $\psi=r$. These formulas appear on the right-hand side of the distance defined in \Cref{section-types} and will give us a criterion for type omitting.

It is also important to note the quantifier complexity of $\theta(\psi,r,\epsilon)$: since $r$ and $\epsilon$ are real numbers, and we cannot assume $r$ is a rational, $r\bone$ and $\epsilon\bone$ are both $\inf^2$ and $\sup^2$. Thus, if $\psi$ is $\sup^\alpha$ or $\inf^\alpha$, we have that $\theta(\psi,r,\epsilon)$ is $\inf^{\max(\alpha+1, 3)}$. 

\begin{definition}
Let $\Omega$ be a weak modulus, $\tA$ a metric $\tau$-structure, and $\alpha>0$ an ordinal. A partial $(\Omega, L)$-type $\Psi(\bx)$ is supported in $\tA$ if there is a definable predicate $P(\bx)$, with $\min I_P=0$, such that
$$\inf_{\bx}P^{\tA}(\bx)=0$$
and for every $\epsilon>0$ and every $\bx\in \tA^{|\bx|}$: 
\begin{center}
if $P(\bx)<\epsilon$ then $\theta^{\tA}(\psi ,r_\psi, \epsilon)(\bx)=0$ for every $(\psi, r_\psi)\in \Psi(\bar x)$.
\end{center}
Moreover, we say that $\Psi(\bx)$ is $(\Omega,\overline\Gamma)$-supported if $P(\bx)$ is a weakly $(\Omega,\Gamma)$ definable predicate.
\end{definition}

The two conditions on $P$ together tell us that $\Psi(\bx)$ is approximately realizable in $\tA$. Moreover, if $P^{\tA}(\bx)=0$, then $\bx$ is a realization of $\Psi$.\par

\begin{remark}
When considering a partial  $(\Omega, \overline L)$-type, or any other partial type made out of definable predicates, this definition is only useful if we add the condition that $\overline{\interior{\Psi(\bx)}}=\Psi(\bx)$, meaning that every predicate in $\Psi$ is the uniform limit of formulas in $\Psi$. In such a situation, we can extend the conclusion for the whole type: $P(\bx)<\epsilon$ then $\theta^\tA(Q, 0, \epsilon)(\bx)=0$ for every $Q\in \Psi(\bar x)$. Note that this is always true when $\Psi(\bx)=(\Omega,\overline{\sup^\alpha})-\operatorname{tp}(\bar a)$ for some $\bar a\in \tA^{<\omega}$. 
\end{remark}

\begin{lemma}\label{char-supported}
Let $\Omega$ be a weak modulus, $\tA$ a metric $\tau$-structure, and $\alpha>0$ a limit ordinal. Let $\Psi(\bx)= (\Omega,\overline{\sup^{<\alpha}})-\operatorname{tp}(\bar a)$ for some $\bar a\in \tA^{<\omega}$. If
\begin{enumerate}[label={($\bigstar$)}]
\item for every $m\geq 1$, there is an $(\Omega, \inf^{<\alpha})$ formula $\phi_m$ and a $\delta_m>0$ such that
\begin{enumerate}
\item there is some $\bx_m\in \tA^{|\bx|}$ with $\phi_m^{\tA}(\bx_m)<\delta_m$;
\item for every $x\in \tA^{|\bx|}$, if $\phi_m^\tA(\bx)<\delta_m$ then $\theta(\psi,r_\psi,2^{-m})^\tA(\bx)=0$ for all $(\psi,r_\psi)\in \interior{\Psi(\bx)}$.
\end{enumerate}
\end{enumerate}
Then $\Psi(\bx)$ is $(\Omega, \overline{\inf^{<\alpha}})$ supported in $\tA$.
\end{lemma}

\begin{proof}
The proof is largely the same as Proposition 12.5 of  \cite{benyaacov2008}. Without loss of generality, we can find numbers $\eta_m$ such that
$$\phi_m^{\tA}(\bx_m)\leq\eta_m<\delta_m\leq 2^{-m}$$
by replacing $\phi_m$ with $\frac{2^{-m}}{\delta_m}\phi_m$ if $\delta_m\geq 2^{-m}$, which is an $\Omega$ formula since $\frac{2^{-m}}{\delta_m}\leq 1$, and choose $\eta_m$ to be rational. Now define the
$$\hat\phi_m(\bx) = \inf_{\by}\max( d_\Omega(\bx,\by)\ \dot-\ 2^{-m}, \phi_m(\by)\ \dot-\ \eta_m).$$

We claim that $\hat\phi_m^{\tA}(\bar a)=0$ for all $m\geq1$. Suppose that $\hat\phi_m^{\tA}(\bar a)\geq t>0$, then $(t\dot-\hat\phi_m,0)\in \interior{\Psi(\bx)}$. Thus, there is some $\bar z\in \tA^{|\bx|}$, such that $d_\Omega(\bx_m,\bar z)\leq 2^{-m}$ and $\hat\phi_m(\bar z)\geq t/2$. However, $\bx_m$ satisfies satisfies the condition inside the infimum for $\bar z$, so $\hat\phi_m^{\tA}(\bar z)=0$. A contradiction. \par

Also, if $\hat\phi_m^\tA(\bx)<\delta_m-\eta_m$, then $\theta^{\tA}(\psi,r_\psi,2^{-m+2})(\bx)=0$ for all $(\psi,r_\psi)\in \interior{\Psi(\bx)}$.\par

Now, consider the $(\Omega,\inf^{<\alpha})$ formulas $\sum_{m=1}^N 2^{-m}\hat\phi_m(\bx)$, which converge uniformly to a weakly-$(\Omega,\inf^{<\alpha})$ definable predicate $P_*=\sum_{m=1}^\infty 2^{-m}\hat\phi_m(\bx)$. Note that $P_*$ now satisfies the following: $P_*^\tA(\ba)=0$ and for every $\epsilon>0$ there is some $\delta>0$ such that if $P_*^\tA(\bx)<\delta_m$ then $\theta(\psi,r_\psi,\epsilon)^{\tA}(\bx)=0$ for all $(\psi,r_\psi)\in \interior{\Psi(\bx)}$.\par

To finish, apply \Cref{Henson-alpha} with $F=P_*$ and $G(\bar x)$ being the least $\epsilon>0$ such that the condition on all the $\theta$ formulas is satisfied. Then $$P(\bar x)=\inf_{\bar y} \alpha P_*(\bar y)+ d_\Omega(\bar x,\bar y)$$ 
is a weakly-$(\Omega,\inf^{<\alpha})$ definable predicate approximated by the $(\Omega,\inf^{<\alpha})$ formulas $$\inf_{\bar y} \sum_{m=1}^N\left( 2^{-m}\alpha_N\hat\phi_m(\by)\right)+d_\Omega(\bar x,\bar y),$$ 
where the $\{\alpha_N\}_{N<\omega}$ are quantifier free $\Omega$ formulas whose uniform limit is $\alpha$. 
\end{proof}

\begin{theorem}\label{type-omitting}
Let $\Omega$ be a weak modulus, and $\alpha>0$ a limit ordinal. Let $\tA$ be a separable metric $\tau$-structure, $Q=\sup_n\sup_{\bx_n}P_n(\bx_n)$ be such that each $P_n$ is a weakly-$(\Omega,\inf^{<\alpha})$ definable predicate, and $Q^\tA=0$. Let $\Psi(\bx)= (\Omega,\overline{\sup^{<\alpha}})-\operatorname{tp}(\bar a)$ for some $\bar a\in \tA^{<\omega}$ which is not $(\Omega, \overline{\inf^{<\alpha}})$ supported in $\tA$. Then there is a separable metric $\tau$-structure $\tB$ such that $Q^\tB=0$ and $\tB$ omits $\Psi(\bx)$.
\end{theorem}

It is enough that $\Psi(\bx)$ be an $(\Omega,\overline{\sup^{<\alpha}})$ type such that $\overline{\interior{\Psi(\bx)}}=\Psi(\bx)$.

\begin{proof}
Write $P_n=\lim \phi_{n,m}$, where $||P_n-\phi_{n,m}||_{\tA}<2^{-m}$. Expand the vocabulary by adding a set $C$ of countably many constant functions, and let $\tau^*$ be the resulting vocabulary.\par

Our key tool follows from the contrapositive of \Cref{char-supported}: there is some psitive integer $m\geq 1$, such that for every  $(\Omega, \inf^{<\alpha})$ formula $\phi(\bx)$ and every $\delta>0$, either $\phi\geq \delta$ everywhere in $\tA$ or there is some $x\in \tA^{|\bx|}$ and some $(\psi,r_\psi)\in \interior{\Psi(\bx)}$ with
$\phi^\tA(\bx)<\delta$ but $\theta^{\tA}(\psi,r_\psi,2^{-m})(\bx)>0$. We fix such an $m$ for the rest of the construction.\par

We will build a set $\Gamma=\bigcup_{k<\omega}\sigma_k$, where $\sigma_k\subset\sigma_{k+1}$ is a finite set of pairs $(\phi, r)$, where $\phi$ is an $\inf^{<\alpha}$ $\tau$-formula, with all of its variables replaced by constants from $C$\footnote{We would like to say that $\phi$ is a $\tau^*$-sentence, as opposed to a $\tau$ formula with all of its variables replaced by constants from $C$. Indeed, this two notions are the same in discrete first-order and infinitary logic, even in finitary continuous logic. In infinitary continuous logic however, adding constants results in a strictly larger language (see Remark 2.4 of \cite{hallback}) with new formulas that we do not want to introduce into the construction.}, and $r\in (I_\phi\setminus\{\min I_\phi\})\cap\mathbb{Q}$. However, we will not require that $\phi$ satisfies $\Omega$ as a modulus of continuity. Moreover, we will guarantee that:

\begin{enumerate}
\item[\textbf{CON:}] $\Gamma$ satisfies the conditions (P1), (P2) and (P3) of \Cref{model-construction}. 
\item[\textbf{SAT:}] for every $k<\omega$, there is a variable assignment that makes $\sigma_k$ satisfiable in $\tA$. That is, we can interpret the constants in $\tA$ so that $\phi^\tA(\bar c)<r$ for all $(\phi(\bar c),r)\in\sigma_k$ under this interpretation.
\item[\textbf{REAL:}] for every tuple of constants $\bar c$ of length $|\bx_n|$, and every $n,m<\omega$, there is some $k<\omega$ such that $(\phi_{n,m}(\bar c), 2^{-m})\in \sigma_k$;
\item[\textbf{OM:}] for every tuple of constants $\bar c$ of length $|\bx|$, there is some $(\psi,r_\psi)\in\interior{\Psi}$, $t>0$ rational, and $k<\omega$ such that $(-\theta(\psi,r_\psi,\epsilon)(\bar c), -t)\in \sigma_k$.
\end{enumerate}
so that we can construct a separable metric $\tau$-structure $\tB$ from $\Gamma$ as in \Cref{model-construction}. Such $\tB$ satisfies $Q^\tB=0$ by construction and omits $\Psi(\bx)$. Indeed, the construction guarantees that the (equivalence classes) of Henkin constants are a countable dense set of the structure, and \textbf{OM} guarantees that no element of $\tB$ within $\epsilon$ of a Henkin constant realizes the type.\par

As for the construction, fix listings $\{t_i\}_{i\in\omega}$ of all the $\tau^*$ terms, $\{\bar c_i\}_{i<\omega}$ of $C^{<\omega}$ such that each element appears infinitely often.  Given $\sigma_k$, use \textbf{SAT} to see $\sigma_k$ as a set of conditions about elements of $\tA$. Then, for each of the finitely many conditions in $\sigma_k$, satisfy one step of (P1)-(P3) such that the resulting set is still satisfiable in $\tA$. For example, if $(\inf_n\psi_n, r)\in \sigma_k$,  then we add $(\psi_n, r)$ to $\sigma_{k+1}$, where $\psi_n$ is the least disjunct we haven't yet added. Since $\tA$ is a metric $\tau$-structure, this is always possible. Also, if $\bar c_k$ is of the right length, we extend the interpretation if needed to add $(\phi_{n,m}(\bar c_k), 2^{-m})$ to $\sigma_{k+1}$ for $n,m\leq k$. We call the resulting set of conditions $\sigma^*_k$.\par

Finally, if $\bar c_k$ has size $|\bar x|$, write $\sigma^*_k=\{(\phi_i(\bar c_k,\bar d),r_i)\}_{i<\ell}$ for some $\ell<\omega$, and where $\bar d$ is the finite set of all other constants mentioned. By scaling if necessary, the finite set of conditions is equivalent to an expression of the form $\phi(\bar c_k,\bar d)<r^*$, where $r^*=\min(r_0,\dotsc, r_{\ell-1})$. By \textbf{SAT}, there is a variable assignment and $s<r^*$ such that $\phi(\bar c_k,\bar d)^\tA=s<r^*$. Let $\epsilon= \frac{r^*-s}{3}$.\par

We want to consider the $\inf^{<\alpha}$ $\tau$-formula $\inf_{\by}\phi(\bx,\by)$. Although we cannot guarantee it satisfies $\Omega$, it does satisfy some modulus of continuity $\Delta$. Fix $\delta>0$ small enough such that $\delta\leq\epsilon$ and  $\Delta(\delta,\dotsc, \delta)<\epsilon$. Now, we focus our attention on the $(\Omega,\inf^{<\alpha})$ formula
$$\psi(\bx) = \inf_{\bar z} \left( \left(\inf_{\by} \phi(\bar z,\by)\ \dot- \ s \right) + d_\Omega(\bx,\bar z) \right)$$
and note that under the variable assignment for $\sigma^*_k$ we have $\psi^\tA(\bar c_k)=0<\delta$. By the contrapositive of \Cref{char-supported}, we can find some  $\bar a\in \tA^{|\bx|}$ such that  $\psi^\tA(\bar a)<\delta$ and $\theta(\psi,r_\psi,2^{-m})^{\tA}(\bar a)\geq t>0$.\par
By construction of $\psi$, there is some $\bar z\in \tA^{|\bx|}$ with $\inf_{\by} \phi^\tA(\bar z,\by)<s+\delta$ and $d_\Omega^\tA(\bar a, \bar z)<\delta$.  Since $\Omega$ is universal, $d(a_i,z_i)^\tA<\delta$ for $i<|\bx|$, which implies that 
$$\Delta(d^\tA(a_0,z_0),\dotsc, d^\tA(a_{|\bx|-1},z_{|\bx|-1}))\leq \Delta(\delta,\dotsc,\delta)<\epsilon.$$
Therefore, $\inf_{\by}\phi(\bar a,\by)^\tA<s+\delta+\epsilon\leq s+2\epsilon$. In turn, we can find some $\bar b\in \tA^{|\by|}$, such that $\phi(\bar a,\bar b)^\tA<s+3\epsilon=r^*$. Thus, by updating our variable assignment, $\sigma_{k+1} = \sigma^*_k\cup\{(-\theta(\psi,r_\psi,2^{-m})(\bar c_k)^{\tA},-t)\}$ is satisfiable in $\tA$ and guarantees \textbf{OM}.
\end{proof}

\section{Back and Forth}
The goal of this section is to begin the study of separable structures in continuous infinitary logic.  We introduce the notion of an $\omega$-presentation of a separable metric structure, give a definition of back-and-forth sets in this setting, and show that back-and-forth sets correspond to isomorphisms (when using a universal weak modulus).

\begin{remark}
Throughout this section, let $\tau$ be a countable continuous vocabulary and $L=L^\mathbb{R}_{\omega_1,\omega}(\tau)$. By Lemma 4.1 of \cite{benyaacov2017} we may assume that $\tau$ is a relational vocabulary.
\end{remark}

\begin{definition}
	\phantom{a}\par
	\begin{enumerate}
		\item We say that a metric $\tau$-structure $\tA$ is \textit{separable} if the metric space $(A,d^\tA)$ is separable.
		\item A sequence $\A= \{a_n\mid n\in\omega\}$ is countable tail-dense sequence in a metric space $(\tA,d)$ if for any $k\in\omega$, the sequence $\{a_n\mid n\geq k\}$ is dense in $(\tA,d$). We do this to correctly handle isolated points in the arguments that follow.
	 \end{enumerate}
 \end{definition}

 \begin{definition}\label{baf-definition}
Let $\mathcal{A}$ and $\mathcal{B}$ be countable tail-dense sequences of separable metric $\tau$-structures $\tA$ and $\tB$ respectively, $\Omega$ a universal modulus for $\tau$, and $t>0$. A \textit{bounded back-and-forth set with bound $t$} is a set $I\subset \mathcal{A}^{<\omega}\times\mathcal{B}^{<\omega}$ such that for $(\overline a,\overline b)\in I$ we have:
\begin{enumerate}
\item $\sup_\phi|\phi^\tA(\overline a)-\phi^\tB(\overline b)|<t$ where $\phi(\overline x)$ varies over all $|\bar a|=|\bar b|$-ary $\Omega$-quantifier free formulas plus all formulas from \Cref{Upower}.iv of the right length\footnote{We want to think of $d^\Omega$ as a quantifer-free formula, but it might only be a uniform limit of quantifer free formulas. However, to ease notation, we will treat it as such for the remainder of this paper};
\item For every $c\in \mathcal{A}$ there is a $d\in \mathcal{B}$ such that the index of $d$ in $\mathcal{B}$ is larger than all the indices in $\overline a,\overline b, c$ and $(\overline ac,\overline bd)\in I$;
\item For every $d\in \mathcal{B}$ there is a $c\in \mathcal{A}$  such that the index of $c$ in $\mathcal{A}$ is larger than all the indices in $\overline a,\overline b, d$ and $(\overline ac,\overline bd)\in I$.
\end{enumerate}
 \end{definition}
 
 \begin{lemma}
Let $\mathcal{A}$ and $\mathcal{B}$ be countable tail-dense sequences of separable metric $\tau$-structures $\tA$ and $\tB$ respectively, and $\Phi:\overline{\mathcal{A}}\cong_\tau\overline{\mathcal{B}}$ is a $\tau$-isomorphism between their completions. Then 
 $$I(\Omega,\Phi,t)=\{(\overline a,\overline b)\mid d_\Omega^\tB(\Phi(\overline a),\overline b)<t\}\subset \A^{<\omega}\times \B^{<\omega}$$ 
 is a bounded back-and-forth set for every  weak modulus $\Omega$, and every $t>0$.
 \end{lemma}
 
\begin{proof}
Suppose $(\overline a,\overline b)\in I$. Then for any $\Omega$-quantifier-free formula $\phi$ we have:
\begin{align*}
|\phi^\tA(\overline a)-\phi^\tB(\overline b)| &\leq |\phi^\tA(\overline a)-\phi^\tB(\Phi(\overline a))| + |\phi^\tB(\Phi(\overline a)) -\phi^\tB(\overline b)|\\
&\leq  0 + |\phi^\tB(\Phi(\overline a)) -\phi^\tB(\overline b)|\\
&\leq  d_\Omega^\tB(\Phi(\overline a),\overline b)\\
&< t
\end{align*}
Note that $|\phi^\tA(\overline a)-\phi^\tB(\Phi(\overline a))|=0$ since $\Phi$ is an isomorphism, and the second-to-last inequality comes from the fact that $\phi$ is an $\Omega$ formula. This shows (1) is satisfied.\par

We now show that (2) also holds, and note that (3) is analogous. Fix some $c\in\A$; let $n=|\overline a|=|\overline b|$, and $r_i=d^\tB(\Phi(a_i),b_i)$ for $i<n$. Then the assumption that $d_\Omega^\tB(\Phi(\overline a),\overline b)<t$ means exactly that
$$\Omega|_n((r_0,\dotsc, r_{n-1}) = \Omega(r_0,\dotsc, r_{n-1},0,0\dotsc)=\Omega|_{n+1}((r_0,\dotsc, r_{n-1},0) < t.$$
Since $\Omega|_{n+1}$ is continuous, there is some $r_n>0$ such that $\Omega|_{n+1}(r_0,\dotsc,r_n)<t.$
By tail-density of $\mathcal{B}$ in $\overline{\mathcal{B}}$, there is some $d\in\mathcal{B}$, whose index is larger than all indices we have seen so far, such that $d^{\overline{\mathcal{B}}}(\Phi(c),d)<r_n$. This implies that $d_\Omega^\tB(\Phi(\overline ac),\overline b d)<t$, so $(\overline ac,\overline bd)\in I$ and we are done.
\end{proof}

The following theorem is a generalization of \cite[Theorem 5.5]{benyaacov2017}. The proof that $\Phi$ is an isomorphism is largely the same as in \cite{benyaacov2017}:

 \begin{theorem}\label{baf-auto}
 Let $\mathcal{A}$ and $\mathcal{B}$ be countable tail-dense sequences of separable metric $\tau$-structures $\tA$ and $\tB$ respectively, $\Omega$ a universal modulus for $\tau$, and $t>0$. If $I\subset \mathcal{A}^{<\omega}\times\mathcal{B}^{<\omega}$ is a bounded back-and-forth set with bound $t$ and  $(\overline a,\overline b)\in I$,  with  $n_0=|\overline a|=|\overline b|$, then there is a $\tau$-isomorphism $\Phi:\overline{\mathcal{A}}\cong_\tau\overline{\mathcal{B}}$ between the completions such that $d_\Omega^\tB(\Phi(\overline a),\overline b)\leq t$.
 \end{theorem}
 
\begin{proof}
We construct two increasing functions $f,g:\omega\to\omega$ as follows: 
\begin{itemize}
\item Let $\overline a^{-1} = \overline a$ and $\overline b^{-1}=\overline b$. 
\item Stage $2s$: Given $(\overline a^{2s-1},\overline b^{2s-1})\in I$, we apply the forth property to $a_s\in \mathcal{A}$, obtaining a $d\in\mathcal{B}$ such that $(\overline a^{2s-1}a_s,\overline b^{2s-1}d)\in I$ and the index of $d$ in $\mathcal{B}$ is larger than all numbers that have appeared so far in the construction. Let $\overline a^{2s}=a^{2s-1}a_s$, and $\overline b^{2s}=\overline b^{2s-1}d$. Also, let $f(s)$ be the index of $d$ in $\mathcal{B}$. The assumption on the index of $d$ makes $f$ increasing.
\item Stage $2s+1$:  Given $(\overline a^{2s},\overline b^{2s})\in I$, we apply the forth property to $b_s\in \mathcal{B}$, obtaining a $c\in\mathcal{A}$ such that $(\overline a^{2s}c,\overline b^{2s}b_s)\in I$ and the index of $c$ in $\mathcal{A}$ is larger than all numbers that have appeared so far in the construction. Let $\overline a^{2s+1}=a^{2s}c$, and $\overline b^{2s+1}=\overline b^{2s}b_s$. Also, let $g(s)$ be the index of $c$ in $\mathcal{A}$. The assumption on the index of $c$ makes $g$ increasing.
\end{itemize}

We can now define $\Phi: \overline{\mathcal{A}}\to \overline{\mathcal{B}}$ as follows: $\Phi(a)=b$ if there is an increasing sequence $m_k$ such that $\lim_k a_{m_k}=a$ and $\lim_k b_{f(m_k)}=b$, and we proceed to verify this is actually a $\tau$-isomorphism:

\begin{claim}\label{baf-1}
If $m_k$ is an increasing sequence such that $\left\{ a_{m_k} \right\}_k$ is Cauchy in $\tA$, then $\left\{ b_{f(m_k)} \right\}_k$ is Cauchy in $\tB$
\end{claim}

\begin{proof}
Fix $\epsilon>0$, and let $M>0$ be such that $\frac{t}{M}<\frac{\epsilon}{2}$. Apply Lemma \ref{Upower} (iii) to find $N_1=N(M,d(x_0,x_1))$, and let $N_2$ be such that if $k>l>N_2$ then $d(a_{m_k},a_{m_l})<\frac{\epsilon}{2}$. \par
Take $k>l>\max\{N_1,N_2\}$ then $n_0+2m_k>n_0+2m_l>N_1$, so 
$$\phi(x_0,\dotsc, x_{n_0+2m_k}) = Md(x_{n_0+2m_l},x_{n_0+2m_k})$$
is an $\Omega$-quantifier-free-formula. As $(\overline a^{2m_k},\overline b^{2m_k})\in I$, we must have that
$$|\phi^\tA(\overline a^{2m_k})-\phi^\tB(\overline b^{2m_k})|<t$$
so
$$|d^\tA(a_{m_l},a_{m_k}) - d^\tB(b_{f(m_l)},b_{f(m_k)})|<\frac{t}{M}<\frac{\epsilon}{2}$$
and since $d^\tA(a_{m_k},a_{m_l})<\frac{\epsilon}{2}$, we obtain that $d^\tB(b_{f(m_l)},b_{f(m_k)})<\epsilon$.
\end{proof}

\begin{claim}\label{baf-2}
If $m_k$ and $n_k$ are increasing sequences such that $\lim_k a_{m_k}=\lim_k a_{n_k}$, then $\lim_k b_{f(m_k)}=\lim_k b_{f(n_k)}$. In particular, $\Phi$ is well-defined.
\end{claim}

\begin{proof}
The proof is the same as in the previous claim. The only difference being that we choose $N_2$ such that if $k>N_2$, then $d(a_{m_k},a_{n_k})<\frac{\epsilon}{2}$.
\end{proof}

\begin{claim}\label{baf-3}
If $\phi(\overline x)$ is an atomic formula and $\overline c\in \overline{\mathcal{A}}^{|\overline x|}$, then $\phi^{\overline{\mathcal{A}}}(\overline c)=\phi^{\overline{\mathcal{B}}}(\Phi(\overline c))$. In particular, $\Phi$ is an isometry, thus injective, and an elementary embedding.
\end{claim}

\begin{proof}
Fix $\epsilon>0$, let $n= |\overline x|$, and write $\overline c = (c_0,\dotsc, c_{n-1})$. Let $M>0$ be such that $\frac{t}{M}<\frac{\epsilon}{3}$. Apply Lemma \ref{Upower} (iii) to find $N_1=N(M,\phi(\overline x))$. Also, fix increasing sequences $\{a_{m^i_k}\mid k\in\omega\}\subset \A$ such that $\lim_{k\to\infty} a_{m^i_k}=c_i$ for $i<n$, and observe that from the previous claims we have that $\lim_{k\to\infty} b_{f(m^i_k)} = \Phi(c_i)$. Finally, using continuity of $\Omega|_n$, let $N_2$ be such that if $k>N_2$ then 
$$\Omega|_n(d^\tA(a_{m^i_k},c_i)\mid i<n)<\frac{\epsilon}{3}$$
and 
$$\Omega|_n(d^\tB(b_{f(m^i_k)},\Phi(c_i))\mid i<n)<\frac{\epsilon}{3}.$$

Fix $\sigma\in\mathbb{N}^n$ such that $\min\sigma > \max\{N_1,N_2\}$ and let $s>\max_{i<n} n_0+2m^i_{\sigma(i)}$. Then
$$\psi(x_0,\dotsc, x_{s}) = M\phi(x_{n_0+2m^0_{\sigma(0)}},\dotsc, x_{n_0+m^{n-1}_{\sigma(n-1)}})$$
is an $\Omega$-formula (even after adding dummy variables with larger indexes). Since $(\overline a^s,\overline b^s)\in I$, we have
$$|\psi^\tA(\overline a^s) - \psi^\tB(\overline b^s)|<t$$
so
$$|\phi^\tA(a_{m^0_{\sigma(0)}},\dotsc,a_{m^{n-1}_{\sigma(n-1)}}) - \phi^\tB(b_{m^0_{f(\sigma(0))}},\dotsc,b_{m^{n-1}_{f(\sigma(n-1))}})|<\frac{\epsilon}{3}.$$

To finish the proof, observe that
\begin{align*}
|\phi^\tA(\overline c)-\phi^\tB(\Phi(\overline c))| &\leq |\phi^\tA(\overline c)-\phi^\tA(a_{m^0_{\sigma(0)}},\dotsc,a_{m_{\sigma(n-1)}^{n-1}})|\\
&+ |\phi^\tA(a_{m^0_{\sigma(0)}},\dotsc,a_{m_{\sigma(n-1)}^{n-1}})-\phi^\tB(b_{f(m^0_{\sigma(0)})},\dotsc,b_{f(m_{\sigma(n-1)}^{n-1})})|\\
&+ |\phi^\tB(b_{f(m^0_{\sigma(0)})},\dotsc,b_{f(m_{\sigma(n-1)}^{n-1})})-\phi^\tB(\Phi(\overline c))|
\end{align*}
where we have already shown that the second term in the right-hand side is less that $\frac{\epsilon}{3}$ (the value of the formula does not change when going to the completion). For the first and third term, observe that $\phi$ is an $\Omega$ formula, so in particular
$$|\phi^\tA(\overline c)-\phi^\tA(a_{m^0_{\sigma(0)}},\dotsc,a_{m_{\sigma(n-1)}^{n-1}})| \leq \Omega|_n(d^\tA(a_{m^i_k},c_i)\mid i<n)<\frac{\epsilon}{3}$$
by our assumption above. The same argument shows that the third term is also less than $\frac{\epsilon}{3}$. Thus, the left-hand side is less than $\epsilon$ and we are done.
\end{proof}

Note these statements still hold if we replace $f$ by $g$ everywhere. This will be important to show surjectivity:

\begin{claim}
Fix $b\in\tB$, If $m_k$ is an increasing sequence such that $\lim_k b_{m_k}=b$, then $\lim_k b_{f(g(m_k))}=b$. In particular, if $a=\lim_k a_{g(m_k)}$, then $\Phi(a)=b$. 
\end{claim}

\begin{proof}
Assume $b$ and $m_k$ are as above. By \Cref{baf-1} for $g$, we have that $\{a_{g(m_k)}\}_k$ is Cauchy in $\mathcal{A}$, and then using \Cref{baf-1} again, but this time for $f$, we get that $\{b_{f(g(m_k))}\}_k$ is Cauchy in $\mathcal{B}$. Thus, it suffices to show that $\lim_{k\to\infty} d^\mathcal{B}(b_{m_k}, b_{f(g(m_k))}) = 0$.\par

Returning to the construction of $f$ and $g$, note the following: at stage $2m_k+1$, we extended the pair $(\overline a^{2m_k},  \overline b^{2m_k})$ by adding $b_{m_k}$ on the $\mathcal{B}$ side, and $a_{g(m_k)}$ to the $\mathcal{A}$ (or rather, we defined $g(m_k)$ to be the index that makes this true). In particular,   $b_{m_k}$ and $a_{g(m_k)}$ appear at coordinate $n_0+2m_k+1$ and $g(m_k)>m_k$. Later, at stage $2g(m_k)$, we extended the sequence by adding  $a_{g(m_k)}$ to the $\mathcal{A}$ and $b_{f(g(m_k))}$ to the $\mathcal{B}$ side, so they both appear at coordinate $n_0+g(m_k)$.\par

Now, fix $\epsilon>0$, and let $M>0$ be such that $\frac{t}{M}<\epsilon$. Use \Cref{Upower} (iii) to find $N=N(M,d(x_0,x_1))$ and let $k>N$. Then $\phi(x_0,\dotsc, x_{n_0+2g(m_k)}) = Md(x_{n_0+2m_k+1}, x_{2g(m_k)})$ is an $\Omega$-formula and therefore
$$|\phi^\tA(\overline a^{2g(m_k)}) -  \phi^\tB(\overline b^{2g(m_k)})| < t$$
so we have that
$$| d^\tA(a_{g(m_k)},a_{g(m_k)}) - d^\tB(b_{m_k},b_{f(g(m_k))})| <\epsilon$$
by our analysis in the previous paragraph. Clearly, the distance on the $\tA$ side is zero, so $d^\tB(b_{m_k},b_{f(g(m_k))}) <\epsilon$ and we are done. 
\end{proof}

\begin{claim}
$d_\Omega^\tB(\Phi(\overline a),\overline b)\leq t$
\end{claim}
\begin{proof}
Write $\overline a = (c_0,\dotsc, c_{n_0-1})$ and $\overline b = (d_0,\dotsc, d_{n_0-1})$. We use $c$ and $d$ to avoid confusion with the following part: fix increasing sequences $\{a_{m^i_k}\mid k\in\omega\}\subset \A$ such that $\lim_{k\to\infty} a_{m^i_k}=c_i$ for $i<n$, and such that
$$k<m^0_k<\dotsc<m^{n_0-1}_k.$$
By \Cref{baf-1} we get that $\lim_{k\to\infty} b_{f(m^i_k)} = \Phi(c_i)$ and $\Phi(\overline a) = (\Phi(c_0),\dotsc, \Phi(c_{n_0-1}))$. By \Cref{Upower} (iv),  
$$d_\Omega^\tB((x_0,\dotsc, x_{n_0-1}), (x_{n_0+m^0_k},\dotsc, x_{n_0+m^{n_0-1}_k}))$$ 
is an $\Omega$-formula, and considered party of the the back-and-forth property by assumption. Thus, we get
$$|d_\Omega^\tA(\overline a, (a_{m^0_k},\dotsc, a_{m^{n_0-1}_k})) - d_\Omega^\tB(\overline b, (b_{f\left(m^0_k \right)},\dotsc, b_{f\left(m^{n_0-1}_k\right)}))|<t.$$
Letting $k$ go to infinity, the first term vanishes by the assumption on the sequences, and the second term approaches $d_\Omega^\tB(\overline b, \Phi(\overline a))$. Thus, $d_\Omega^\tB(\overline b, \Phi(\overline a))\leq t$ and we are done.
\end{proof}
This completes the proof. 
\end{proof}
 
\begin{proposition}\label{elem-baf}
Let $\Omega$ be a universal weak-modulus for $\tau$, $\mathcal{A}$ and $\mathcal{B}$ be countable tail-dense sequences of separable metric $\tau$-structures $\tA$ and $\tB$ respectively, and $t>0$. If $\tA$ and $\tB$ satisfy the same $\Omega$-sentences; then the set 
$$I(\Omega, t)=\{(\overline a,\overline b) \mid  \sup_\phi |\phi^\mathcal{A}(\overline a)-\phi^\mathcal{B}(\overline b)| <t\}\subset \A^{<\omega}\times \mathcal{B}^{<\omega},$$ 
where $\phi$ ranges over all $\Omega$-formulas of the appropriate length and without any additional parameters, is nonempty and is a bounded back-and-forth set.
\end{proposition}

\begin{proof}
Fix $(\overline a,\overline b)\in I$, and note that condition (1) is trivial. The assumptions about $\mathcal{A}$ and $\mathcal{B}$ guarantee that $(\langle\rangle, \langle\rangle)\in I$.\par
Suppose for contradiction that (2) doesn't hold; then for some $c\in\A$ and every $d\in \mathcal{B}$, there is an $\Omega$-formula $\phi_{c,d}(\overline x,y)$ such that
$$|\phi^\tA_{c,d}(\overline a, c) - \phi^\tA_{c,d}(\overline b,d)| \geq t.$$
Without loss of generality, $\phi^\tA_{c,d}(\overline a,c)=0$ and  $\phi^\tB_{c,d}(\overline b,d)\geq t$. Then consider the $\Omega$-formula
$$\psi(\overline x) = \inf_{y}\sup_{d\in\mathcal{B}} \phi_{c,d}(\overline x,y)$$
and note that the second quantifier is countable since $\mathcal{B}$ is countable. By construction, $\psi(\overline a)^\tA=0$, with $y=c$, but $\psi(\overline b)^tB\geq t$. Thus, $(\overline a,\overline b)\not\in I$, a contradiction. Note that condition (3) is analogous.
\end{proof}

\begin{remark}
The reader might be tempted to instead consider the set 
$$J(\Omega, t)=\{(\overline a,\overline b) \mid  \sup_\phi |\phi^{(\mathcal{A},\overline a)}-\phi^{(\mathcal{B},\overline b)}| <t\}\subset \A^{<\omega}\times \mathcal{B}^{<\omega},$$ 
where $\phi$ varies over all sentences in the language expanded by the requisite number of constants. However, this is not necessarily a back-and-forth set. The reasoning is expanded in Remark 2.4 of \cite{hallback}, but in essence it is possible for $\operatorname{tp}_\A(\overline a)=\operatorname{tp}_\mathcal{B}(\overline b)$, while $(\A,\overline a)\not\equiv_\tau (\mathcal{B}, \overline b)$ once we add constants naming the parameters. 
\end{remark}

 \begin{corollary}\label{elem-to-auto}
Let $\Omega$ be an universal weak-modulus for $\tau$. If  $\tA$ and $\tB$ are separable metric $\tau$-structures that satisfy the same $\Omega$-sentences, then they are isomorphic (as $\tau$-structures). In particular, they are isometric.
 \end{corollary}

\section{Definability, Scott Ranks, and Scott Predicates}
We now use back-and-forth sets to show that the closures of automorphism orbits in separable metric structure are definable, and in turn use this fact to construct Scott sentences, and Scott predicates, for separable metric structures.\par
Although the following results are true for sentences, and thus true for predicates, we will see in the next section that working with predicates gives us a correct and robust notion of Scott rank.

 \begin{lemma}\label{closures-same-vals}
Let $\Omega$ be a universal weak-modulus for $\tau$, $\tA$ be a separable metric $\tau$-structure, and  $\overline a,\overline b\in \tA^{n}$ for some $n<\omega$. Then $\overline b\in\overline{Aut_\tA(\overline a)}$ if, and only if, $\phi^\A(\overline a)=\phi^\A(\overline b)$ for all $\Omega$-formulas without parameters. 
 \end{lemma}
 
 \begin{proof}
Assume $\overline b\in\overline{Aut_\tA(\overline a)}$, and fix $\epsilon>0$. By continuity of $\Omega|_{n}$ under the sup norm, there is some $\delta>0$ such that $\lVert\overline r\rVert_\infty<\delta$, then $\Omega|_n(\overline r)<\epsilon$. By assumption, we can find $\overline c\in Aut_\tA(\overline a)$ such that $d^\tA_n(\overline b,\overline c)<\delta$, so $d_\Omega^\tA(\overline b,\overline c)<\epsilon$ as $\Omega|_n$ is non-decreasing. This implies that $|\phi^\A(\overline a)-\phi^\A(\overline b)|<\epsilon$ for any $\Omega$ formula $\phi$, and the result follows.\par
 
For the other direction: fix $t>0$ and a tail dense sequence $\A$ of $\tA$ containing $\overline a$ and $\overline b$. Let $I(\Omega,t)$ be as in \Cref{elem-baf} with $\mathcal{B}=\mathcal{A}$. By assumption, $(\overline a,\overline b)\in I(\Omega,t)$. Therefore, by \Cref{baf-auto}, there is an automorphism $\Phi:\tA\to \tA$ with $d_\Omega^\tA(\Phi(\overline a),\overline b)\leq t$. As this is true for all $t>0$, we get 
$$\inf_{\overline c\in Aut_\A(\overline a)} d_\Omega^\tA(\overline b,\overline c)=0,$$ 
so $\overline b\in\overline{Aut_\A(\overline a)}$. 
\end{proof}

 \begin{theorem}\label{def-orbits}
Let $\Omega$ be a universal weak-modulus, $\tA$ be a separable metric $\tau$-structure, and $\overline a\in \tA^{<\omega}$. Then the closure of the automorphism orbit of $\overline a$ in $\tA$ is definable by an  $\Omega$-formula. Meaning that there is an $\Omega$-formula $\psi_{\bar a}(\bar x)$ such that $\psi_{\bar a}^\tA(\bar x) = d_\Omega^\tA(\bar x,\overline{Aut_\tA(\bar a)})$.
\end{theorem}
  
 \begin{proof}
 By \Cref{closures-same-vals}, for each $\overline b\in \A^{|\overline a|}$ with $\overline b\not\in\overline{Aut_\A(\overline a)}$, there is an $\Omega$-formula without parameters $\phi_{\overline a,\overline b}(\overline x)$ where $\overline a$ and $\overline b$ differ. In fact, observe that for any $\Omega$-formula $\phi$,
 $$|\phi^\A(\overline a) - \phi^\A(\overline b)|\leq d_\Omega^\tA(\overline a,\overline b)<\infty.$$
So for any $\overline b\not\in\overline{Aut_\A(\overline a)}$, the set 
$$\{|\phi^\A(\overline a) - \phi^\A(\overline b)|\mid \phi\ \text{is an $\Omega$-formula without parameters}\}$$
is bounded and has a supremum $v_{\overline b}>0$. Moreover, the supremum is attained since we have the $\sup_n$ quantifier, so let $\phi_{\overline a,\overline b}(\overline x)$ be a formula with the largest discrepancy. Without loss of generality, further assume that  $\phi^\A_{\overline a,\overline b}(\overline a)= 0$, and $\phi^\A_{\overline a,\overline b}(\overline b)= v_{\overline b}$. Since $\A$ is countable, we can write
$$\psi_{\overline a}(\overline x) = \sup_{\overline b\in \A^{|\overline a|},\ \overline b\not\in\overline{Aut_\A(\overline a)}} \phi_{\overline a,\overline b}(\overline x).$$
We claim that $\psi_{\overline a}(\overline x)$ defines $Aut_\A(\overline a)$. It is clear by \Cref{closures-same-vals}, that the zeroset of $\psi_{\overline a}(\overline x)$ is $Aut_\A(\overline a)$, so it remains to show that $\psi_{\overline a}(\overline x)$ defined the distance to the closure of the automorphism orbit.\par
Suppose $\overline c\in \A^{|\overline a|}$, and $\phi_{\overline a}(\overline c)<t$ for some $t>0$. By construction, 
$$|\phi^\A(\overline a) - \phi^\A(\overline c)|<t,$$
so $(\overline a,\overline c)\in I(\Omega, t)$ as defined in \Cref{elem-baf}. Thus, by \Cref{baf-auto}, there is an automorphism $\Phi$ of $\A$ such that $d^\A_n((\Phi(\overline a),\overline c)\leq t$. Hence, setting $\delta=\epsilon$ in \Cref{define} is enough. 
 \end{proof}

\begin{corollary}\label{def-sets}
 Let $\tA$ be a separable metric $\tau$-structure, and $n\geq 1$ a positive integer. A nonempty set $X\subset \tA^n$ is definable if, and only if, it is metrically closed and automorphism invariant. In fact, if $\Omega$ is a universal modulus for $\tau$, then $X$ is definable by an $\Omega$-formula.
\end{corollary}

\begin{proof}
The backwards direction is the meaningful one. Suppose we are given $X\subset \tA^n$, which is metrically closed and automorphism invariant. Note that $(A^n, d_{\Omega|_n}^\tA)$ is a separable, complete, bounded metric space (continuity of $\Omega|_n$ is used here). Thus, the same is true of $(X, d_{\Omega|_n}^\tA)$ as it is metrically closed. Fix a countable dense subset $E\subset X$. y \Cref{def-orbits}, $\mathcal{O}_{\bar e}=\overline{Aut_\tA(\bar e)}$ is definably by an $\Omega$-formula $\phi_{\bar e}(\bar x)$ for every $\bar e\in E$. Since $X$ is automorphims invariant and metrically closed, $\bigcup_{\bar e\in E} \mathcal{O}_{\bar e}\subset X$.\\\par

We claim that the $\Omega$-formula $\psi(\bar x) = \inf_{\bar e\in E} \phi_{\bar e}(\bar x)$ defines $X$. Given $\bar y\in X$ and $\epsilon>0$, we can find some $\bar e\in E$ so that $d_\Omega^\tA(\bar y,\bar e)<\epsilon$. Hence $\psi^\tA(\bar y)=0$. Next, suppose that for some $\bar y\in \tA^n$ with $\psi^\tA(\bar y)<\epsilon$ for some $\epsilon>0$. Thus, there is some $\bar e\in E$ with $\phi_{\bar e}^\tA(\bar y)<\epsilon$. This implies that there is some $\bar e'\in \mathcal{O}_{\bar e}\subset X$ with $d_\Omega^\tA(\bar e',\bar y)<\epsilon$. This completes the proof by \Cref{define}.
\end{proof}
 
\begin{definition}
A Scott Predicate for a separable metric $\tau$-structure $\tA$ is a definable predicate $S$, without free variables and with $\min I_S=0$, such that $S^{\overline{\mathcal{B}}}=0$ if, and only if, $\tA\cong \overline{\mathcal{B}}$ for all separable metric $\tau$-structures $\overline{\mathcal{B}}$. We define the notion of Scott Sentence similarly. 
\end{definition}

 \begin{theorem}\label{SP-exists}
 Every separable metric $\tau$-structure has an $L^\mathbb{R}_{\omega_1,\omega}(\tau)$ Scott Predicate. In fact, if $\Omega$ is a universal modulus for $\tau$, then every metric $\tau$-structure has a Scott Predicate made out of $\Omega$-formulas in a precise sense.
 \end{theorem}

\begin{proof}
Our proof is the continuous version of the argument given in Chapter II of \cite{montalban2021}: the idea is to write a predicate $S$ with the following property: if $\overline{\mathcal{B}}$ is any separable metric $\tau$-structure, then $S^{\overline{\mathcal{B}}}=0$ guarantees that $I(\Omega,t)$, as defined in \Cref{elem-baf}, is a nonempty bounded back-and-forth set for any (equivalently all) $t>0$, between countable tail-dense sequences of the respective structure. Such a predicate must have compact range, so we choose $[0,1]$ as our range. \par

Fix a separable metric $\tau$-structure $\tA$, a universal modulus $\Omega$, and let $\A$ be a countable tail-dense sequence of $\tA$. For any $\overline a\in \A^{<\omega}$, let $P_{\overline a}(\overline x)$ be as in \Cref{def-orbits}   (i.e. a definable predicate that outputs the distance to the closure of the automorphism orbit). We will also fix $\Omega$-formulas $\{\psi_{\bar a,n}\}_{n\in\omega}$ such that $P_{\bar a}$ is the uniform limit of this sequence and $\lVert P_{\bar a}-\psi_{\bar a,n}\rVert\leq 2^{-n}$.\par

We also set $\psi_{\langle\rangle}$ to be the identically zero function, and let $\Theta=\Theta(\Omega)$ be the (countable) set of quantifier-free $\Omega$-formulas plus all the formulas  from \Cref{Upower}.iv (here we use that $\tau$ is countable and there are only countably many connectives). Then define the sentences:

$$S_n = \min\left( \bone,\  \sup_{\overline a\in \A^{<\omega}}\ 
	\sup_{x_0,\dotsc, x_{|\overline a|}} 
		\Bigg[
			\max
				\Bigg(
					\chi^1_{\bar a}(\bx),\chi^2_{\bar a,n}(\bx y), \chi^3_{\bar a,n}(\bx,y)
				\Bigg)\ 
			\dot-\ \psi_{\overline a, n}(\overline x)
		\Bigg]
\right)$$

where
\begin{align*}
\chi^1_{\bar a}(\bx)&= \sup_{\phi\in\Theta}\ \max\left( \sup_{r\in\mathbb{Q}, r<\phi^\A(\overline a)}r\bone \dot-\ \phi(\overline x), \sup_{r\in\mathbb{Q}, r>\phi^\A(\overline a)}\phi(\overline x)\ \dot-\ r\bone\right) \\
\chi^2_{\bar a,n}(\bx y) &= \sup_{c\in\A}\inf_y \psi_{\overline a c, n}(\overline x y)\\
\chi^3_{\bar a,n}(\bx y) &= \sup_y \inf_{c\in\A}  \psi_{\overline a c,n }(\overline x y)
\end{align*}

Some observations: $\bone$ is as defined in \Cref{the-language}. The formulas mentioned might take arbitrarily large values, so taking the minimum with one is needed; although note that the formula is always nonegative by construction. The quantifier $\sup_{r<\phi^\A(\overline a),r\in\mathbb{Q}}$ is needed as the language only allows multiplication by rationals, but it does not affect the quantifier complexity of the whole sentence. The quantifiers $\sup_{b\in \A}$ and $\inf_{b\in\A}$ are countable suprema and infima; this is why we fix a countable tail-dense sequences. Using the prenex-normal form rules from \Cref{prenex-rules}, we can move all the quantifiers to the front. Finally, note that the whole expression inside the supremums is an $\Omega$-formula. \par

To get the Scott Predicate, we put $S=\lim_{n\to\infty} S_n$. This limit exists for any separable metric $\tau$-structure $\tB$ since $\psi_{\bar a,n}\to P_{\bar a}$ uniformly for every $\bar a$. Therefore, $S$ can also be written exactly as the $S_n$, but replacing $\psi_{\bar a,n}$ with $P_{\bar a}$ everywhere for every  $\bar a$\par

It remains to show that $S^\mathcal{B}=0$ implies $I(\Omega,t)$ is a bounded back-and-forth set for any $t>0$. In fact, the second term in the $\min$ has value $0$ exactly when  $I(\Omega,t)$ is a bounded back-and-forth set, but we have to take the $\min$ to make it a valid sentence in the language. If $(\overline a,\overline b)\in I$, then $P_{\bar a}(\bar b)<t$, and since $S^\mathcal{B}=0$, this implies that the maximum of the three $\chi_{\bar a}$ expressions is $\leq P_{\bar a}(\bar b)<t$. Thus, all three of the $\chi_{\bar a}$ expressions are $<t$, but now $\chi^1_{\bar a}(\bar b)<t$ guarantees (1) of \Cref{baf-definition}, $\chi^2_{\bar a}(\bar b)<t$ guarantees (2), and $\chi^3_{\bar a}(\bar b)<t$ guarantees (3). It is nonempty as the case $\overline a = \langle\rangle$ implies $(\langle\rangle, \langle\rangle)\in I$.
\end{proof}

 \begin{corollary}
 Every separable metric $\tau$-structure has an $L^\mathbb{R}_{\omega_1,\omega}(\tau)$ Scott Sentence. In fact, if $\Omega$ is a universal modulus for $\tau$, then every metric $\tau$-structure has a Scott Sentence made out of $\Omega$-formulas in a precise sense.
 \end{corollary}

\begin{proof}
Either note that $S=\limsup S_n = \liminf S_n$ is a $\tau$-sentence or begin with formulas that define the automorphism orbits and define $S$  directly.
\end{proof}

\section{A Robust Scott Rank For Separable Structures}
 \begin{definition}
Let $\Omega$ be a universal modulus for $\tau$, and $\Gamma$ be a subclass of $L^\mathbb{R}_{\omega_1,\omega}(\tau)$. A separable metric $\tau$-structure $\tA$ is densely $(\Omega,\Gamma)$-atomic if there is a countable tail-dense sequence $\A$ of $\tA$ such that the metric closure of the automorphism orbit of every tuple $\overline a\in \A^{<\omega}$ is $(\Omega, \inf^\alpha)$-definable in $\tA$ without parameters.   
\end{definition}

\begin{lemma}
Let $\mathcal{A}$ be a countable tail-dense sequence for a separable metric $\tau$-structure $\tA$, and $\Omega$ a universal modulus for $\tau$. Then there is an ordinal $\alpha<\omega_1$ such that the metric closure of the automorphism orbit of every tuple $\bar a\in \A^{<\omega}$ is $(\Omega, \inf^\alpha)$-definable in $\tA$ without parameters.   
\end{lemma}

\begin{proof}
Since $\A$ is countable, so is $\A^{<\omega}$. The metric closure of the automorphism orbit of every tuple $\bar a\in \A^{<\omega}$ is definable by some definable predicate, which we may assume is weakly-$(\Omega, \inf^{\alpha_{\bar a}})$ for some countable ordinal $\alpha_{\bar a}$. Then, let $\alpha=\sup\{\alpha_{\bar a}\mid \bar a\in \A^{<\omega}\}$. Note then that $\alpha$ is a countable ordinal since $\omega_1$ has uncountable cofinality. 
\end{proof}

 \begin{definition}
Given a universal modulus $\Omega$, the (parameterless) $\Omega$-Scott rank of a separable metric $\tau$-structure $\tA$, written $SR^\Omega(\tA)$, is the least ordinal $\alpha$ such that $\tA$ is densely $(\Omega,\Gamma)$-atomic.
\end{definition}

\begin{lemma}\label{SP-complexity}
Let $\mathcal{A}$ be a countable tail-dense sequence for a separable metric $\tau$-structure $\tA$, $\Omega$ a universal modulus for $\tau$, and $\alpha$ an ordinal. If the closure of every automorphism orbit of every tuple coming from $\mathcal{A}$ is $(\Omega, \inf^\alpha)$-definable without parameters, then $\tA$ has a $\sup^{\alpha+1}$ Scott predicate. 
 \end{lemma}
 
 \begin{proof}
This is obtained by counting the quantifier complexity of the sentences in \Cref{SP-exists}. Note that $\chi_{\bar a}^1$ is $(\Omega,\sup^1)$ since the formulas are quantifier free. Also, since $\alpha>0$, we can choose the formulas $\psi_{\bar a, n}$ so that they satisfy $\Omega$ by \Cref{good-approx}.
\end{proof}

The following lemma gives us a sufficient condition for a separable structure to be $(\Omega, \inf^\alpha)$-atomic:

 \begin{lemma}\label{types-to-orbits}
Let $\mathcal{A}$ be a countable tail-dense sequence for a separable metric $\tau$-structure $\tA$, $\Omega$ a universal modulus for $\tau$, and $\alpha>0$ a countable ordinal. If every maximal $(\Omega,\overline{\sup^\alpha})$-type realized in $\mathcal{A}$ is supported in $\tA$ by a weakly-($\Omega,\inf^{\alpha})$-definable predicate without parameters, then the closure of every automorphism orbit in $\mathcal{A}$  is $(\Omega, \inf^\alpha)$-definable without parameters.
 \end{lemma}
 
\begin{proof}
For every $\overline a\in \A^{<\omega}$, let $P_{\overline a}(\overline x)$ be the weakly-$(\Omega, \inf^{\alpha})$ predicate that supports 
$$\Phi_{\bar a} = (\Omega, \overline{\sup\phantom{}^\alpha})\text{-}\operatorname{tp}_\tA(\bar a)$$
and for $\langle\rangle$ we choose the identically zero function. We claim that these formulas actually define the respective (closures of) automorphism orbits.\par

First, we show that $P^\tA_{\bar a}(\bar a)=0$. Note that $-P_{\bar a}$ is weakly $(\Omega,\sup^\alpha)$ and $\inf_{\bar x} P^\tA_{\bar a}(\bar x)=0$ by definition. Therefore, given $\epsilon>0$, fix some $\bar x\in A^{|\bar a|}$ with $P^\tA_{\bar a}(\bar x)<\epsilon$. There is some tuple $\bar y\in \tA^{|\bar a|}$ with $d_\Omega^\tA(\bar y,\bar x)<\epsilon$ and $|P_{\bar a}^\tA(\bar a)-P_{\bar a}^\tA(\bar y)|<\epsilon$. Therefore, since $P_{\bar a}$ satisfies $\Omega$, we get $ P^\tA_{\bar a}(\bar a)<3\epsilon$.

Also, for any $\bar b\in A^{|\bar a|}$, we have that $\Phi_{\bar a}$ decides the value of $-P_{\bar b}(\bx)$. Thus, if $P^\A_{\overline a}(\overline b)<t$, then we also have $P^\A_{\overline b}(\overline a)<t$. In particular, if $P^\A_{\overline a}(\overline b)=0$, then $b$ realizes $\Phi_{\bar a}$ and $P^\A_{\overline b}(\overline a)=0$ as well.\par

We also note that \Cref{elem-baf} implies that $P^\A_{\overline a}(\overline b)=0$ for any  $\overline b\in \overline{Aut_\A(\overline a})$.\par

It remains to show that $P_{\overline a}(\overline x)$ satisfies the $\epsilon-\delta$ condition in \Cref{define}, which we do via a back-and-forth argument: consider the set
$$K(\Omega, t) = \{ (\overline a,\overline b)\mid P^\tA_{\overline a}(\overline b)<t\}\subset \left( \A^{<\omega}\right)^2.$$
We claim that $K(\Omega, t)$ is a nonmepty, bounded back-and-forth set for all $t>0$. The set is nonempty as $(\bar a,\bar a)\in K(\Omega,t)$. Next, let $(\overline a,\overline b)\in K(\Omega, t)$ and we check the three conditions in the definition of back-and-forth set:
\begin{enumerate}
\item Fix a quantifier-free $\tau$-formula $\phi(\bx)$ respecting $\Omega$. Suppose $\phi^\tA(\bar a)=r$, so $(\phi,r)\in \Phi_{\bar a}$. By our assumption on $(\bar a,\bar b)$, there is some $0<t'<t$ such that $\theta(\phi, r,t')(\bar b)^\tA=0$. Therefore, since  $\phi(\bx)$ respects $\Omega$, it follows that $|\phi^\tA(\bar b)-\phi^\tA(\bar a)|\leq t'<t$.

\item Assume for contradiction that there is some $c\in\A$ such that for every $d\in \A$ we have $P_{\bar a c}^\tA(\bar b d)\geq t$
then consider the weakly-$(\Omega, \sup^\alpha)$ predicate
$$Q(\bar x) =  \sup_y -P_{\bar a c}(\bx y) $$
note that $P_{\bar a c}(\bar a c)^\tA=0$, so $Q\in \Phi$. Also, by construction, $Q^\tA(\bar b)\leq -t$. Therefore, $P_{\bar b}(\bar a)^\tA=P_{\bar a}(\bar b)^\tA\geq t$, which is a contradiction to $(\bar a,\bar b)\in K(\Omega,t)$. 
\item Similar to (2).
\end{enumerate}
To finish the proof, by \Cref{baf-auto}, $P^\A_{\overline a}(\overline b)<t$ implies that there is some $\overline c\in   \overline{Aut_\A(\overline a})$ with $d^{\Omega,\A}(\overline c,\overline b)\leq t$, and we are done.
\end{proof}

 \begin{theorem}[Robustness]
Let be a separable metric $\tau$-structure $\tA$, $\Omega$ a universal modulus for $\tau$, and $\alpha>0$ a countable ordinal. Fix $\mathcal{A}$, a countable tail-dense sequence for $\tA$. Then, the following are equivalent:
\begin{enumerate}
\item the closure of every automorphism orbit of  every $\bar a\in \mathcal{A}^{<\omega}$ is $(\Omega, \inf^{\alpha})$-definable without parameters;
\item every $(\Omega,\overline{\sup^{\alpha}})$  type realized in $\mathcal{A}$ is supported in $\tA$ by a weakly-($\Omega,\inf^{\alpha})$-definable predicate without parameters
\end{enumerate}
Moreover, they imply that
\begin{enumerate}
\item[(3)] $\tA$ has an $(\Omega,\sup^{\alpha+1})$ Scott predicate. That is, a predicate of the form  $$Q=\sup_n\sup_{\bx_n}P_n(\bx_n),$$ such that each $P_n$ is a weakly-$(\Omega,\inf^{\alpha})$ definable predicate
\end{enumerate}
 \end{theorem}

\begin{proof}
(1)$\to$(2): Fix $\bar a\in \A^{<\omega}$. By assumption, there is a weakly-$(\Omega, \inf^\alpha)$ predicate $P_{\bar a}$ satisfying that $P_{\bar a}^\tA(\bar x)=d_\Omega^\tA(\bar x,\overline{Aut_\tA(\bar a)})$. It is a routine check of the definition to show that $P_{\bar a}$ supports $(\Omega,\overline{\sup^{\alpha}})-\operatorname{tp}_\tA(\bar a)$. For partial types, note that there are at most countably many realizations of the type in $\A^{<\omega}$. We can construct a weakly-$(\Omega, \inf^\alpha)$ predicate that supports the type by taking the infimum over all the definition of automorphism orbits for tuples that realize the type.\par

(2)$\to$(1) is \Cref{types-to-orbits}.\par

(1)$\to$(3): is \Cref{SP-complexity}.
\end{proof}

We also obtain the following partial converse:

\begin{theorem}
If $\alpha>0$ is a countable limit ordinal and $\tA$ has an $(\Omega,\sup^{\alpha})$ Scott predicate, then $SR^\Omega(\tA)\leq\alpha$. In fact, for every $\bar a\in\tA^{<\omega}$, there is some $\beta<\alpha$ such that $\overline{Aut_{\tA}(\bar a)}$ is definable (without parameters) by a weakly-$(\Omega,\inf^\beta)$ definable predicate.
\end{theorem}

\begin{proof}
The proof is by contradiction. Suppose $\tA$ has an $(\Omega,\sup^{\alpha})$ Scott Predicate $Q$ as described, but there is some $\bar a\in \tA^{<\omega}$ such that $\overline{Aut_{\tA}(\bar a)}$ is not $(\Omega,\inf^{<\alpha})$-definable (i.e. it is not $(\Omega,\inf^\beta)$-definable for all $\beta<\alpha$). Note that \Cref{types-to-orbits} works in this setting, so we must have that the $(\Omega,\overline{\sup^{<\alpha}})-\operatorname{tp}(\bar a)$ is not $(\Omega, \overline{\inf^{<\alpha}})$ supported in $\tA$. Then, by \Cref{type-omitting}, there is  a separable metric $\tau$-structure $\tB$ such that $Q^\tB=0$ and $\tB$ omits $\Psi(\bx)$. Thus, there cannot be an isomorphism between $\tA$ and $\tB$, which contradicts our assumption $Q$.
\end{proof}


\printbibliography
\end{document}